\documentclass[12pt, twoside]{article}
\usepackage{amscd,latexsym,amsthm,amsfonts,amssymb,amsmath,amsxtra}
\usepackage[all]{xy}

\usepackage[mathscr]{eucal}
\usepackage[pdftex,bookmarks,colorlinks,pdfmenubar]{hyperref}
\usepackage{times}
\usepackage{enumerate}

\def\titlerunning#1{\gdef\titrun{#1}}
\makeatletter
\def\author#1{\gdef\autrun{\def\and{\unskip, }#1}\gdef\@author{#1}}
\def\address#1{{\def\and{\\\hspace*{18pt}}\renewcommand{\thefootnote}{}%
\footnote {#1}}%
\markboth{\autrun}{\titrun}}
\makeatother
\def\email#1{e-mail: #1}
\def\subjclass#1{{\renewcommand{\thefootnote}{}%
\footnote{\emph{Mathematics Subject Classification (2010):} #1}}}
\def\keywords#1{\par\medskip
\noindent\textbf{Keywords.} #1}

\newcommand{\BA}{{\mathbb {A}}}

\newcommand{\BC}{{\mathbb {C}}}

\newcommand{\BZ}{{\mathbb {Z}}}

\newcommand{\CA}{{\mathcal {A}}}

\newcommand{\CE}{{\mathcal {E}}}
\newcommand{\CF}{{\mathcal {F}}}

\newcommand{\CO}{{\mathcal {O}}}
\newcommand{\CP}{{\mathcal {P}}}

\newcommand{\CS}{{\mathcal {S}}}

\newcommand{\CU}{{\mathcal {U}}}

\newcommand{\CZ}{{\mathcal {Z}}}

\newcommand{\Fd}{{\mathfrak {d}}}

\newcommand{\Fg}{{\mathfrak {g}}}

\newcommand{\Fm}{{\mathfrak {m}}}
\newcommand{\Fn}{{\mathfrak {n}}}

\newcommand{\Fr}{{\mathfrak {r}}}

\newcommand{\RU}{{\mathrm {U}}}

\newcommand{\Asai}{{\mathrm{As}}}

\newcommand{\cusp}{{\mathrm{cusp}}}
\newcommand{\Cent}{{\mathrm{Cent}}}

\newcommand{\disc}{{\mathrm{disc}}}

\newcommand{\el}{{\mathrm{ell}}}

\newcommand{\Gal}{{\mathrm{Gal}}}
\newcommand{\GL}{{\mathrm{GL}}}

\newcommand{\reg}{{\mathrm{reg}}}
\newcommand{\Res}{{\mathrm{Res}}}

\newcommand{\simp}{{\mathrm{sim}}}
\newcommand{\SL}{{\mathrm{SL}}}

\newcommand{\SU}{{\mathrm{SU}}}

\newcommand{\sm}{{\mathrm{ss}}}

\newcommand{\st}{{\mathrm{st}}}
\newcommand{\Span}{{\mathrm{Span}}}
\newcommand{\subr}{{\mathrm{subr}}}

\newcommand{\tr}{{\mathrm{tr}}}

\newcommand{\ud}{\,\mathrm{d}}

\newcommand{\udl}{\underline}
\newcommand{\wt}{\widetilde}
\newcommand{\wh}{\widehat}

\newcommand{\cpair}[1]{\left\{{#1}\right\}}
\newcommand{\ppair}[1]{\left( {#1} \right)}

\def\bks{{\backslash}}

\def\diag{{\rm diag}}

\newtheorem{thm}{Theorem}[section]
\newtheorem{cor}[thm]{Corollary}
\newtheorem{lem}[thm]{Lemma}
\newtheorem{prop}[thm]{Proposition}
\newtheorem {conj}[thm]{Conjecture}
\newtheorem {ques/conj}[thm]{Question/Conjecture}

\numberwithin{equation}{section}

\begin{document}

 \titlerunning{}

\title{On the Non-vanishing of the Central Value of Certain $L$-functions: Unitary Groups}

\author{Dihua Jiang
\and
Lei Zhang
}
\date{}

\maketitle

\address{Dihua Jiang: School of Mathematics, University of Minnesota, Minneapolis, MN 55455, USA; \email{dhjiang@math.umn.edu}
\and
Lei Zhang: Department of Mathematics,
National University of Singapore,
Singapore 119076;
\email{matzhlei@nus.edu.sg}
}




\begin{abstract}
Let $\pi$ be an irreducible cuspidal automorphic representation of a quasi-split unitary group $\RU_\Fn$ defined over a number field $F$.
Under the assumption that $\pi$ has a generic global Arthur parameter, we establish the non-vanishing of the central value of $L$-functions, $L(\frac{1}{2},\pi\times\chi)$, with a certain automorphic character $\chi$ of $\RU_1$, for the case of $\Fn=2,3,4$, and for the general $\Fn\geq 5$
by assuming Conjecture \ref{conj-subr} on certain refined properties of global Arthur packets. In consequence, we obtain some simultaneous
non-vanishing results for the central $L$-values by means of the theory of endoscopy.

\end{abstract}


\section{Introduction}

Let $F$ be a number field, and $G_n^*$ be an $F$-quasisplit unitary group that is either $\RU_{n,n}$ or $\RU_{n+1,n}$.
For an irreducible cuspidal automorphic representation $\pi$ of $G_n^*(\BA)$, where $\BA$ is the ring of adeles of $F$, we assume that $\pi$
has a generic global Arthur parameter (\cite{A13} and \cite{Mk15}) and consider the non-vanishing problem of $L(\frac{1}{2},\pi\times\chi)$,
the central value of the tensor product $L$-function of $\pi$ with an automorphic character $\chi$ of $\RU_1$.

When $G_1^*$ is an $F$-form of $\RU_2$, the non-vanishing of $L(\frac{1}{2},\pi\times\chi)$ can be detected by the non-vanishing of
certain periods of automorphic representations closely related to $\pi$, via the well-known Waldspurger formula (\cite{W85} and also \cite{Gross}).
Hence the non-vanishing in this situation can be established with certain arithmetic assumptions on the
archimedean local places and/or ramified finite local places (\cite{H93} for instance). For higher rank unitary groups, the non-vanishing of
the central value $L(\frac{1}{2},\pi\times\chi)$ has been formulated as a basic input towards the investigation of many interesting and important
problems related to the theory of motives, arithmetic geometry, and the theory of $p$-adic $L$-functions. We refer to some recent papers of M. Harris
(\cite{H13} and \cite{H16}) for more detailed account of the significance of such non-vanishing property.

The objective of this paper is to understand such a non-vanishing problem in a framework containing:
\begin{enumerate}
\item Fourier coefficients of automorphic forms associated to unipotent orbits as developed in \cite{J14} and \cite{JL-Cogdell},
\item the global zeta integrals of tensor product type as developed in \cite{JZ14} and \cite{JZ-BF},
\item the theory of twisted automorphic descents as developed in \cite{JZ-BF} and \cite{JZ-Howe}, and
\item the endoscopic classification of the discrete spectrum as in \cite{A13}, \cite{Mk15}, and \cite{KMSW}.
\end{enumerate}
The following conjecture represents a basic version of such a non-vanishing assertion as expected.

\begin{conj}\label{conj-nvn1}
For any given generic global Arthur parameter $\phi$ of $G_n^*$, there exists an automorphic character $\chi$ of $\RU_1$ such that
for every automorphic member $\pi$ in the global Arthur packet $\wt{\Pi}_\phi(G_n^*)$ associated to $\phi$,
the central value $L(\frac{1}{2},\pi\times\chi)$ of the tensor product $L$-function of $\pi$ and $\chi$ is nonzero.
\end{conj}

Note that in this paper, by a generic global Arthur parameter $\phi$, we mean a generic {\sl elliptic} global Arthur parameter
$\phi\in\wt{\Phi}_2(G_n^*)$ as explained in Section \ref{ssec-dsap}.

In reality, one may consider a more refined version of Conjecture \ref{conj-nvn1} by putting certain ramification restrictions on $\chi$ based on the
arithmetic nature of $\pi$
or ask for how many such characters $\chi$ the non-vanishing assertion still holds.
We will not get into those issues here.

By means of the theory of endoscopy, Conjecture \ref{conj-nvn1} has the following consequence on simultaneous non-vanishing of central $L$-values.

\begin{cor}\label{cor-smnv}
Let $\phi=(\tau_1,1)\boxplus\cdots\boxplus(\tau_r,1)$ be a generic global Arthur parameter of $G_n^*$ with $\tau_i$ being conjugate self-dual,
irreducible cuspidal automorphic representations of $\GL_{a_i}(\BA_E)$, where $E$ is a quadratic field extension of $F$, and $\sum_{i=1}^r a_i$ is equal to
either $2n$ or $2n+1$, depending on $G_n^*$. Then Conjecture \ref{conj-nvn1} implies that
there exists an automorphic character $\chi$ of $\RU_1$ such that
$$
L(\frac{1}{2},\tau_1\times\chi)\cdot L(\frac{1}{2},\tau_2\times\chi)\cdots L(\frac{1}{2},\tau_r\times\chi)\neq 0.
$$
\end{cor}

As explained in Section \ref{sec-unu1}, the condition for $\tau_1,\cdots,\tau_r$ to make a generic global Arthur parameter $\phi$ can be checked through
the local ramification and sign conditions.

The first main result of this paper is to prove Conjecture \ref{conj-nvn1} for $F$-quasisplit unitary groups: $\RU_{1,1}$, $\RU_{2,1}$ and $\RU_{2,2}$.
\begin{thm}\label{th-U22}
Conjecture \ref{conj-nvn1} is true for $F$-quasisplit unitary groups: $\RU_{1,1}$, $\RU_{2,1}$ and $\RU_{2,2}$.
\end{thm}

This theorem will be proved in Section \ref{sec-n4}. We remark that at this moment, we are not able to extend certain technical arguments in the proof for those lower rank cases
to the higher rank situation, due to the complication of the $F$-rational structure of the Fourier coefficients associated to certain unipotent orbits
involved in the arguments in the proof in Section \ref{sec-n4}.

However, with the assumption of the following very reasonable conjecture on a refined structure on certain global Arthur packets,
the other main result of this paper proves Conjecture \ref{conj-nvn1} for the general situation.

Take $G_n$ to be a unitary group $\RU_{n+1,n-1}$ defined over $F$ with $F$-rank $n-1$.
As discussed in Section \ref{ssec-fcpt}, the partition $\udl{p}=[(2n-1)1]$
corresponds to the $F$-stable subregular unipotent orbit of $G_n^*:=\RU_{n,n}$, which is the largest possible unipotent orbit in $G_n$.

\begin{conj}\label{conj-subr}
Let $\phi$ be a generic global Arthur parameter of $\RU_{n,n}$ that is $G_n$-relevant. Then there exists an automorphic member $\pi_0$ in the
global Arthur packet $\wt{\Pi}_\phi(G_n)$ with the property that $\pi_0$ has a nonzero Fourier coefficient associated to the subregular partition
$\udl{p}=[(2n-1)1]$.
\end{conj}

Note that a global Arthur parameter $\phi$ of $G_n^*$ is called {\sl $G_n$-relevant} if the global Arthur packet $\wt{\Pi}_\phi(G_n)$ of $G_n$ is not empty.
A well-known conjecture of F. Shahidi (\cite{Sh90}) asserts that for any quasi-split reductive algebraic group defined over a local field,
each tempered local $L$-packet has a generic member. The global version of the Shahidi conjecture is that if any global Arthur packet of
quasi-split classical groups $G^*$ defined over a number field $F$ is associated to a generic global Arthur parameter, then it has a member with a non-zero Whittaker-Fourier coefficient. This global conjecture was verified
via the automorphic descents of Ginzburg-Rallis-Soudry (\cite{GRS11}) in Section 3 of \cite{JL-Cogdell} for
quasisplit classical groups based on the endoscopic classification of the discrete spectrum of J. Arthur (\cite{A13}) and of C.-P. Mok (\cite{Mk15}).
As explained in Section \ref{ssec-fcpt}, for the quasi-split classical group $G^*_n$, the Whittaker-Fourier coefficients of automorphic representations are associated to the regular unipotent orbit, which is of course the largest possible unipotent orbit for $G^*_n$. However, the largest unipotent orbit
in the non-quasi-split classical group $G_n$ corresponds to the subregular partition $\udl{p}=[(2n-1)1]$ of $2n$, which is associated to the subregular
unipotent orbit of the quasi-split classical group $G_n^*$. Hence Conjecture \ref{conj-subr} asserts that any global Arthur packet of $G_n$ with
a generic, $G_n$-relevant global Arthur parameter should contain a member with a non-zero Fourier coefficient associated to the largest possible unipotent
orbit of $G_n$. In this sense, Conjecture \ref{conj-subr} can be regarded as a natural extension of the global version of the Shahidi conjecture from quasi-split classical groups to general classical groups. We refer to Conjecture \ref{conj-main} for a general situation of unitary groups.

The second main result of this paper is as follows.

\begin{thm}\label{th-main}
Assume that Conjecture \ref{conj-subr} holds. Then Conjecture \ref{conj-nvn1} is true for all $F$-quasisplit unitary groups.
\end{thm}

Some comments on the structure of the proofs of Theorems \ref{th-U22} and \ref{th-main} are in order, which also explain the structure of the rest of this paper.

After introducing in Section \ref{sec-FCDS} the basics about unitary groups, the discrete spectrum and endoscopic classifications, and Fourier
coefficients associated to unipotent orbits or partitions as needed for this paper, we make Conjecture \ref{conj-main} for general unitary groups,
which is defined by a non-degenerate finite dimensional Hermitian vector space over a quadratic field extension $E$ of the number field $F$, and
is denoted by $G_n$ in general there and for the rest of this paper.
This conjecture
asserts the existence and the structure of the Fourier coefficients associated to the largest possible unipotent orbits of $G_n$, for
irreducible cuspidal automorphic representations of $G_n(\BA)$ with generic global Arthur parameters. Conjecture \ref{conj-main} is the natural extension
of the global version of the Shahidi conjecture to the most general situation of those unitary groups. Of course, there is a similar version for general
reductive groups, about which we omit the discussion here. With help of the {\sl Generic Summand Conjecture} (\cite[Conjecture 2.3]{JZ-BF}),
we deduce, by the {\sl one direction of the global Gan-Gross-Prasad conjecture} that is proved in \cite[Theorem 5.5]{JZ-BF},
and by the assumption of Conjecture \ref{conj-main}, the non-vanishing of the central value of general tensor product $L$-functions for unitary groups in
Theorem \ref{th-nva}.

It is desirable to prove the non-vanishing of central $L$-values without the assumption of Conjecture 2.3 of \cite{JZ-BF},
or even without the assumption of Conjecture \ref{conj-subr} or more generally, Conjecture \ref{conj-main}. At this moment we can only do so for
$\RU_{1,1}$, $\RU_{2,1}$ and $\RU_{2,2}$, which is Theorem \ref{th-U22} and will be treated in Section \ref{sec-n4}.

In Section \ref{sec-unu1}, we follow the main ideas in the proof in Section \ref{sec-CGAP} and treat the special case where Conjecture \ref{conj-main}
reduces to Conjecture \ref{conj-subr}.
In this case, the tensor product $L$-functions are $L(s,\pi\times\chi)$ with $\pi$ being irreducible cuspidal automorphic representation of $\RU_{n+1,n-1}$
with generic global Arthur parameter $\phi$ of $\RU_{n,n}$ that is $\RU_{n+1,n-1}$-relevant, and $\chi$ being an automorphic character of $\RU_1$.
Meanwhile, the partition involved in Conjecture \ref{conj-main} for the situation, which is now Conjecture \ref{conj-subr}, is the subregular partition.
By Proposition 7.4 of \cite{JZ-BF}, Conjecture 2.3 of \cite{JZ-BF} holds in this case.
Hence we only take the assumption of Conjecture \ref{conj-subr} in our discussion for
Section \ref{sec-unu1}, where we prove Theorem \ref{th-main}.

With the endoscopic classification for general unitary groups (\cite{KMSW}) in mind,
it is enough to treat the non-vanishing of the central values $L(\frac{1}{2},\pi\times\chi)$ for $F$-quasisplit unitary groups. For technical reasons, we
treat even unitary groups (in Section \ref{ssec-Fneven}) and odd unitary groups (in Section \ref{ssec-Fnodd}), separately. In order to apply the assumption
of Conjecture \ref{conj-subr} to the situation under consideration, which is about $F$-quasisplit unitary groups, we prove two technical propositions
(Propositions \ref{prop-2n} and \ref{prop-2n1}). For any given generic global Arthur parameter $\phi$ of $G_n^*$, one is able to find a (generic) global
Arthur parameter $\phi_2$ of a non-quasisplit $\RU_2$ if $G_n^*=\RU_{n,n}$, and $\phi_1$ of $\RU_1$ if $G_n^*=\RU_{n+1,n}$, respectively, such that
the isobaric sums $\phi\boxplus\phi_2$ and $\phi\boxplus\phi_1$, in the corresponding cases, are generic global Arthur parameters of the non-quasisplit
unitary group $\RU_{n+2,n}$, which is an $F$-rational form of the $F$-quasisplit unitary group $G^*_{n+1}$, with $F$-rank $n$.
Now, we are able to use the assumption of Conjecture \ref{conj-subr} for $\RU_{n+2,n}$.
Hence Theorem \ref{th-main} is given
by Theorems \ref{th-nveven} and \ref{th-nvodd} according to the even and odd quasisplit unitary groups, respectively. In consequence, we prove in Section \ref{ssec-mnv}
the simultaneous non-vanishing for general linear groups (Theorem \ref{th-nvgl}, i.e. Corollary \ref{cor-smnv}) and for endoscopy groups (Theorem \ref{th-nvug}), using the endoscopic classification of the discrete spectrum of quasisplit unitary groups (\cite{Mk15}).

In Section \ref{sec-n4}, we take up an approach that is different from that in Section \ref{sec-CGAP}. However, this approach only
treats the quasisplit unitary groups $\RU_{1,1}, \RU_{2,1}$ and $\RU_{2,2}$ without the assumption of Conjecture \ref{conj-subr}
and Conjecture 2.3 of \cite{JZ-BF}.
We first prove the non-vanishing for $\RU_{2,2}$ (Theorem \ref{thm:main-U}). The cases of $\RU_{1,1}$ and $\RU_{2,1}$ can be deduced in Section \ref{ssec-lrug}
from Theorem \ref{thm:main-U} by means of the theory of endoscopy.
By the {\sl one direction of the global Gan-Gross-Prasad conjecture} that is proved in \cite[Theorem 5.5]{JZ-BF}, it is enough to prove Proposition \ref{prop-u4subr}, which asserts
that every generic cuspidal automorphic representation $\pi$ of $\RU_{2,2}$ has a non-zero Fourier coefficient associated to the subregular partition
$\udl{p}_\subr=[31]$. The proof of Proposition \ref{prop-u4subr} is a long and explicit calculation of Fourier coefficients associated to unipotent orbits
which are not the maximal one in the wave front set of the irreducible cuspidal automorphic representation $\pi$ under consideration. We refer to Section \ref{sec-proofu4} for the detailed explanation of the structure of the proof. Because of certain technical reasons, we are not able to carry out
such an argument for higher rank unitary groups, which would lead to the proof of Conjecture \ref{conj-subr2n}.
We note that the proof for $\RU_{2,2}$ is direct, without referring to Conjecture \ref{conj-subr} for
$\RU_{4,2}$. Hence the proof in Section \ref{sec-n4} is in different nature from what is represented in Section \ref{sec-unu1}.
However, it seems that the use of Conjecture \ref{conj-subr} in the proof of such non-vanishing results is more conceptual to the theory.
We will come back in our future work to discuss our understanding of Conjecture \ref{conj-subr} or even more generally Conjecture \ref{conj-main}.

Finally, the authors would like to thank M. Harris for useful email exchanges related to the topic of this paper and for his encouragement for us to write
up this paper, and also to thank the referees for the helpful comments.

\section{Fourier Coefficients and Discrete Spectrum}\label{sec-FCDS}

We recall the basics of the endoscopy structure, global Arthur packets, and the Fourier coefficients of automorphic forms on unitary groups from
\cite{JZ-BF}, \cite{Mk15} and \cite{KMSW}. The notation will follow from our previous work \cite{JZ-BF}, since the results from \cite{JZ-BF} will play indispensable roles in
the proof of the main results of this paper.

\subsection{Unitary groups}\label{ssec-ug}
Let $F$ be a number field and $E=F(\sqrt{\varsigma})$ be a quadratic field extension of $F$,
with $\varsigma$ a non-square in $F^\times$. The Galois group
$\Gamma_{E/F}=\Gal(E/F)=\left<\iota\right>$ is generated by a unique non-trivial element $\iota$. \label{pg:iota}
Let $(V,q)$ be an $\Fn$-dimensional non-degenerate Hermitian vector space over $E$. Denote by $G_n=\RU(V,q)$ the unitary group, with $n=[\frac{\Fn}{2}]$.
Let $G_n^*=\RU(V^*,q^*)$ be an $F$-quasisplit group of the same type.
Then $G_n$ is a pure inner form of $G_n^*$ over the field $F$ (\cite{V93} and \cite{GGP12}).

Let $(V_0,q)$ be the $F$-anisotropic kernel of $(V,q)$ with dimension $\Fd_0=\Fn-2\Fr$, where the $F$-rank $\Fr=\Fr(G_n)$ of $G_n$ is
the same as the Witt index of $(V,q)$.
Let $V^{+}$  be a maximal totally isotropic subspace of $(V,q)$, with $\{e_{1},\dots,e_{\Fr}\}$ being its basis.
Choose $E$-linearly independent vectors $\{e_{-1},\dots,e_{-\Fr}\}$ in $(V,q)$
such that
$
q(e_{i},e_{-j})=\delta_{i,j}
$
for all $1\leq i,j\leq\Fr$.
Denote by $V^{-}=\Span\{e_{-1},\dots,e_{-\Fr}\}$ the dual space of $V^+$.
Then $(V,q)$ has the following polar decomposition
$$
V=V^{+}\oplus V_{0}\oplus V^{-},
$$
where $V_{0}=(V^{+}\oplus V^{-})^{\perp}$ is an $F$-anisotropic kernel of $(V,q)$.
We choose an orthogonal basis $\{e'_{1},\dots,e'_{\Fd_0}\}$ of $V_{0}$ with
$
q(e'_{i},e'_{i})=d_{i},
$
where $d_{i}$ is nonzero for all $1\leq i\leq \Fd_0$. Set $G_{d_0}=\RU(V_{0},q)$ with $d_0=[\frac{\Fd_0}{2}]$,
which is anisotropic over $F$ and is regarded as an $F$-subgroup of $G_n$, canonically. In this case, for convenience,
we may say that $G_n$ is a unitary group of type $\RU_{\Fr+\Fd_0,\Fr}$, with $\Fn=\Fd_0+2\Fr$. We may denote by $\CU_{\Fr+\Fd_0,\Fr}$
the set of all unitary groups of such type that are considered in this paper.

We put the above bases together in the following order to form a basis of $(V,q)$:
\begin{equation}\label{bs}
e_{1},\dots,e_{\Fr},e'_{1},\dots,e'_{\Fd_0},e_{-\Fr},\dots,e_{-1}.
\end{equation}
According to the basis in \eqref{bs}, we fix the following totally isotropic flag in $(V,q)$:
$$
\Span\{e_{1}\}\subset\Span\{e_{1},e_{2}\}\subset
\cdots\subset
\Span\{e_{1},\dots,e_{\Fr}\},
$$
which defines a minimal parabolic $F$-subgroup $P_0$.
Moreover, $P_0$ contains a maximal $F$-split torus $S$, consisting of elements
$$
\diag\{t_{1},\dots,t_{\Fr},1,\dots,1,{t}^{-1}_{\Fr},\dots,t^{-1}_{1}\},
$$
with $t_i\in F^\times$ for $i=1,2,\cdots,\Fr$. The centralizer  $Z(S)$ of $S$ in $G_n$ is $\Res_{E/F}S\times G_{d_0}$, which is the Levi subgroup
of $P_0$ and gives the Levi decomposition:
$
P_0=(\Res_{E/F}S\times G_{d_0})\ltimes N_0,
$
where $N_0$ is the unipotent radical of $P_0$.
Also, with respect to the basis in \eqref{bs}, the group $G_n$ is also defined by the Hermitian matrix $J_{\Fr}^\Fn$, which is defined inductively by
\begin{equation} \label{eq:J}
J_{\Fr}^\Fn=\begin{pmatrix}
&&1\\&J_{\Fr-1}^{\Fn-2}&\\1&&
\end{pmatrix}_{\Fn\times\Fn}
\text{ and }
J_{0}^{\Fd_0}=\diag\{d_{1},\dots,d_{\Fd_0}\}.
\end{equation}

Let $_{F}\!\Phi(G_n,S)$ be the root system of $G_n$ over $F$.
Let $_{F}\!\Phi^{+}(G_n,S)$ be the positive roots corresponding to the minimal parabolic $F$-subgroup $P_0$,
and $_{F}\!\Delta=\{\alpha_{1},\dots,\alpha_{\Fr}\}$ a set of simple roots in $_{F}\!\Phi^{+}(G_n,S)$.
Note that the root system $_{F}\!\Phi(G_n,S)$ is non-reduced of type $BC_{\Fr}$ if $2\Fr<\Fn$;
and is of type $C_{\Fr}$, otherwise.

For a subset $J\subset\{1,\dots,\Fr\}$, let $_{F}\!\Phi_{J}$ be the root subsystem  of $_{F}\!\Phi(G_n,S)$
generated by the simple roots $\{\alpha_{j}\colon j\in J\}$.
Let $P_{J}=M_{J}U_{J}$ be the standard parabolic $F$-subgroup of $G_n$,
whose Lie algebra consists of all roots spaces $\Fg_{\alpha}$ with $\alpha\in {_{F}\!\Phi^{+}(G_n,S)}\cup {_{F}\!\Phi_{J}}$.
For instance,  let $\hat{i}=\{1,\dots,\Fr\}\setminus\{i\}$,
and then $P_{\hat{i}}=M_{\hat{i}}N_{\hat{i}}$ the standard maximal parabolic $F$-subgroup of $G_n$,
which stabilizes the rational isotropic space $V^+_i$,
where $V^{\pm}_{i}:=\Span\{e_{\pm1},\dots,e_{\pm i}\}$.
Here $N_{\hat{i}}$ is the unipotent radical of $P_{\hat{i}}$ and
the Levi subgroup $M_{\hat{i}}$ is isomorphic to $G_{E/F}(i)\times G_{n-i}$. Following the notation of \cite{A13} and \cite{Mk15},
$G_{E/F}(i)$ denotes the Weil restriction of $E$-group $\GL_{i}$ restricted to $F$.
Write  $V_{(i)}=(V^{+}_{i}\oplus V^{-}_{i})^{\perp}$ and hence $V_{(\Fr)}=V_0$ is the $F$-anisotropic kernel of $(V,q)$.

\subsection{Discrete spectrum and Arthur packets}\label{ssec-dsap}
Let $\BA=\BA_F$ be the ring of adeles of the number field $F$. Denote by $\CA_\disc(G_n)$ the set of equivalence classes of irreducible
unitary representations $\pi$ of $G_n(\BA)$ occurring in the discrete spectrum $L^2_\disc(G_n)$ of $L^2(G_n(F)\bks G_n(\BA)^1)$, when $\pi$ is restricted to
$G_n(\BA)^1$. Also denote by $\CA_\cusp(G_n)$ for the subset of $\CA_\disc(G_n)$, whose elements occur in the cuspidal spectrum $L^2_\cusp(G_n)$.
The theory of endoscopic classification for classical groups $G_n$ is to parametrize the set $\CA_\disc(G_n)$ by means of the global Arthur parameters.
The global Arthur parameters can be realized as certain automorphic representations of general linear groups. We recall the theory of endoscopic classification
from the work of Arthur (\cite{A13}),
the work of Mok (\cite{Mk15}) and the work of Kaletha, Minguez, Shin, and White (\cite{KMSW}).

Note that the unitary group $G_n$ and its $F$-quasisplit pure inner form $G_n^*$ share the same
$L$-group ${^LG_n^*}={^LG_n}$.
Following \cite{A13}, \cite{Mk15} and \cite{KMSW}, with $N=\Fn$ in the case of unitary groups, we denote by $\wt{\CE}_\simp(N)$ the set of the equivalence classes of simple twisted
endoscopic data, which are represented by the triple $(G_n^*,s,\xi)$, where $G_n^*$ is an $F$-quasisplit classical group, $s$ is a semi-simple
element as described in \cite[Page 11]{A13} and \cite[Page 16]{Mk15}, and $\xi$ is the $L$-embedding
$
{^LG_n^*}\rightarrow {^LG_{E/F}(N)}.
$
Note that $\xi=\xi_{\chi_\kappa}$ depends on $\kappa=\pm1$.
As in \cite[Page 18]{Mk15}, for a simple twisted endoscopic datum $(G_n^*, \xi_{\chi_\kappa})$ of $G_{E/F}(N)$,
the sign $(-1)^{N-1}\cdot\kappa$ is called the {\sl parity} of the datum.
The set of global Arthur parameters for $G_n^*$ is denoted by $\wt{\Psi}_2(G_n^*,\xi)$, or simply by $\wt{\Psi}_2(G_n^*)$ if the
$L$-embedding $\xi$ is well understood in the discussion.

We simply recall from \cite[Section 2.2]{JZ-BF} the more explicit description of the global Arthur parameters for the situation in the current paper.
First, the set of the conjugate self-dual, elliptic, global Arthur parameters for $G_{E/F}(N)$, which is denoted by $\wt{\Psi}_\el(N)$,
consists of global parameters of the following type:
\begin{equation}\label{ellgap}
\psi^N=\psi_1^{N_1}\boxplus\cdots\boxplus\psi_r^{N_r}
\end{equation}
with $N=\sum_{i=1}^rN_i$. The formal summands $\psi_i^{N_i}$ are {\sl simple} parameters of the form
$
\psi_i^{N_i}=\mu_i\boxtimes\nu_i,
$
with $N_i=a_ib_i$, where $\mu_i=\tau_i\in\CA_\cusp(G_{E/F}(a_i))$ and $\nu_i$ is a $b_i$-dimensional representation of $\SL_2(\BC)$.
Following the notation used in our previous paper \cite{J14}, we also set
$
\psi_i^{N_i}=(\tau_i,b_i),
$
for $i=1,2,\cdots,r$. The global parameter $\psi^N$ is called {\sl elliptic} if the decomposition of $\psi^N$ into the simple parameters is
of multiplicity free, i.e. $\psi_i^{N_i}$ and $\psi_j^{N_j}$ are not equivalent if $i\neq j$
in the sense that either $\tau_i$ is not equivalent to $\tau_j$ or $b_i\neq b_j$.
The global parameter $\psi^N$ is called {\sl conjugate self-dual} if each simple parameter $\psi_i^{N_i}$
occurring in the decomposition of $\psi^N$ is conjugate self-dual in the sense that $\tau_i$ is conjugate self-dual. An irreducible
cuspidal automorphic representation $\tau$ of $G_{E/F}(a)$ is called {\sl conjugate self-dual} if $\tau\cong\tau^*$, where
$\tau^*=\iota(\tau)^\vee$ the contragredient of $\iota(\tau)$ with $\iota$ being the non-trivial element in $\Gamma_{E/F}$.
A global parameter $\psi^N$ in $\wt{\Psi}_\el(N)$ is called {\sl generic} if $b_i=1$ for $i=1,2,\cdots,r$.
The set of generic, conjugate self-dual, elliptic, global Arthur parameters for $G_{E/F}(N)$ is denoted by $\wt{\Phi}_\el(N)$. Hence elements $\phi^N$ in
$\wt{\Phi}_\el(N)$ are of the form:
\begin{equation}\label{ellggap}
\phi^N=(\tau_1,1)\boxplus\cdots\boxplus(\tau_r,1).
\end{equation}
When $r=1$, the global parameters $\psi^N$ are called {\sl simple}. The corresponding sets are denoted by $\wt{\Psi}_\simp(N)$ and
$\wt{\Phi}_\simp(N)$, respectively. It is clear that the set $\wt{\Phi}_\simp(N)$ is in one-to-one correspondence with the set
of equivalence classes of the conjugate self-dual, irreducible cuspidal automorphic representations of $G_{E/F}(N)(\BA_F)$.
By \cite[Theorem 1.4.1]{A13} and \cite[Theorem 2.4.2]{Mk15}, for a simple parameter $\phi=\phi^a=(\tau,1)$ in $\wt{\Phi}_\simp(a)$
there exists a unique endoscopic datum $(G_\phi,s_\phi,\xi_\phi)$, such that the parameter $\phi^a$ descends to a global parameter
for $(G_\phi,\xi_\phi)$ in sense that there exists an irreducible automorphic representation $\pi$ in $\CA_2(G_\phi)$, whose
Satake parameters are determined by the Satake parameters of $\phi^a$.

When $G_\phi=\RU_{E/F}(a)$ is the $F$-quasisplit unitary group associated to the $a$-dimensional Hermitian vector space,
the $L$-embedding carries a sign $\kappa_a$, which
determines the nature of the {\sl base change} from the unitary group $\RU_{E/F}(a)$ to $G_{E/F}(a)(\BA_E)$.
By \cite[Theorem 2.5.4]{Mk15}, the (partial) $L$-function
$
L(s,(\tau,1),\Asai^{\eta_{(\tau,1)}})
$
has a (simple) pole at $s=1$ with the sign $\eta_{(\tau,1)}=\kappa_a\cdot(-1)^{a-1}$
(see also \cite[Theorem 8.1]{GGP12} and \cite[Lemma 2.2.1]{Mk15}).
Then the irreducible cuspidal automorphic representation $\tau$ or equivalently the simple generic parameter $(\tau,1)$
is called {\it conjugate orthogonal} if $\eta_{(\tau,1)}=1$ and {\it conjugate symplectic} if $\eta_{(\tau,1)}=-1$, following the terminology of
\cite[Section 3]{GGP12} and \cite[Section 2]{Mk15}.
Here
$L^S(s,(\tau,1),\Asai^+)$ is the (partial) Asai $L$-function of $\tau$ and $L^S(s,(\tau,1),\Asai^-)$ is the (partial) Asai $L$-function of
$\tau\otimes\omega_{E/F}$, where $\omega_{E/F}$ is the quadratic character associated to $E/F$ by the global class field theory.
The sign of a simple global Arthur parameter $\psi=\psi^{ab}=(\tau,b)\in\wt\Psi_2(ab)$ is calculated in \cite[Section 2.2]{JZ-BF},
following \cite[Section 2.4]{Mk15}.

More generally, from
\cite[Section 1.4]{A13} and \cite[Section2.4]{Mk15}, for any parameter $\psi^N$ in $\wt{\Psi}_\el(N)$, there is a twisted elliptic
endoscopic datum $(G,s,\xi)\in\wt{\CE}_\el(N)$ such that the set of the global parameters $\wt{\Psi}_2(G,\xi)$ can be identified as
a subset of $\wt{\Psi}_\el(N)$. We refer to \cite[Section 1.4]{A13}, \cite[Section 2.4]{Mk15}, and \cite[Section 1.3]{KMSW} for
more constructive description of the parameters in $\wt{\Psi}_2(G,\xi)$.
The elements of $\wt{\Psi}_2(G_n^*,\xi)$, with $N=\Fn$ and $n=[\frac{\Fn}{2}]$, are of the form
\begin{equation}\label{aps}
\psi=(\tau_1,b_1)\boxplus\cdots\boxplus(\tau_r,b_r).
\end{equation}
Here $N=N_1+\cdots+N_r$ and $N_i=a_i\cdot b_i$, and $\tau_i\in\CA_\cusp(G_{E/F}(a_i))$ and $b_i$ represents the $b_i$-dimensional
representation of $\SL_2(\BC)$. Note that each simple parameter $\psi_i=(\tau_i,b_i)$ belongs to $\wt{\Psi}_2(G_{n_i}^*,\xi_i)$ with
$n_i=[\frac{N_i}{2}]$, for $i=1,2,\cdots,r$; and for $i\neq j$, $\psi_i$ is not equivalent to $\psi_j$. The parity for $\tau_i$ and $b_i$
is discussed as above. The subset of generic elliptic global Arthur parameters in $\wt{\Psi}_2(G_n^*,\xi)$ is denoted by
$\wt{\Phi}_2(G_n^*,\xi)$, whose elements are in the form of \eqref{ellgap}.

\begin{thm}[\cite{A13}, \cite{Mk15}, \cite{KMSW}]\label{ds}
For any $\pi\in\CA_\disc(G_n)$, there is a $G_n$-relevant global Arthur parameter $\psi\in\wt{\Psi}_2(G_n^*,\xi)$,
such that $\pi$ belongs to the global Arthur packet, $\wt{\Pi}_{\psi}(G_n)$, attached to the global Arthur parameters $\psi$.
\end{thm}

In the rest of this paper, we mainly consider the generic global Arthur parameters of $G_n^*$.
In order to simplify the notation, we may identify the global Arthur parameter $\phi$ in $\wt{\Phi}_2(G_n^*,\xi)$ with the global Arthur
parameter $\phi^N$ with $N=\Fn$ in $\wt{\Phi}_\el(N)$ when the $L$-embedding $\xi$ is assumed.

Also we recall from \cite{JZ-BF} and also from \cite{GGP12} the notation of the local and global Vogan packets. For the given $F$-quasisplit
unitary group $G_n^*$ and its $F$-pure inner forms $G_n$, the global Arthur packet $\wt{\Pi}_\phi(G_n)$ is empty if the
generic global Arthur parameter $\phi$ in $\wt{\Phi}_2(G_n^*,\xi)$ is not $G_n$-relevant. Then the global Vogan packet associated to $\phi$ is defined to
be
\begin{equation}\label{gvp}
\wt{\Pi}_\phi[G_n^*]:=\cup_{G_n}\wt{\Pi}_\phi(G_n).
\end{equation}
The local Vogan packets are defined in the same way, which are denoted by $\wt{\Pi}_{\phi_v}[G_n^*]$, for any generic local Arthur parameters of $G_n^*$.

\subsection{Fourier coefficients associated to partitions}\label{ssec-fcpt}
We simply recall from \cite[Section 2.3]{JZ-BF} (or \cite{J14}) about the Fourier coefficients of automorphic representations.

For an $F$-quasisplit unitary group $G_n^*$ as defined in Section \ref{ssec-ug},
we denote by $\CP(\Fn)$ the set of
all partitions of $N=\Fn$. Note that $\CP(\Fn)$ parametrizes the set of all $F$-stable nilpotent adjoint orbits in the
Lie algebra $\Fg_n^*(F)$ of $G_n^*(F)$, and hence each partition $\udl{p}\in\CP(\Fn)$ defines an $F$-stable nilpotent adjoint orbit $\CO_{\udl{p}}^\st$.
For an $F$-rational orbit $\CO_{\udl{p}}\subset\CO_{\udl{p}}^\st$, the datum $(\udl{p},\CO_{\udl{p}})$ determines
a datum $(V_{\udl{p}},\psi_{\CO_{\udl{p}}})$
for defining Fourier coefficients as explained in \cite{J14} and \cite[Section 2.3]{JZ-BF}. Here $V_{\udl{p}}$ is a unipotent subgroup of $G_n^*$ and
$\psi_{\CO_{\udl{p}}}$ is a non-degenerate character of $V_{\udl{p}}(\BA)$, which is trivial on $V_{\udl{p}}(F)$ and determined by a given non-trivial character $\psi_F$ of $F\bks\BA$.

For an automorphic form $\varphi$ on $G_n^*(\BA)$, the $\psi_{\CO_{\udl{p}}}$-Fourier coefficient of $\varphi$ is defined by the
following integral:
\begin{equation}\label{fcg}
\CF^{\psi_{\CO_{\udl{p}}}}(\varphi)(g)
:=
\int_{V_{\udl{p}}(F)\bks V_{\udl{p}}(\BA)}\varphi(vg)\psi_{\CO_{\udl{p}}}^{-1}(v)\ud v.
\end{equation}
Let $N_{G_n^*}(V_{\udl{p}})^\sm$ be the connected component of the semi-simple part of the normalizer of the subgroup $V_{\udl{p}}$ in $G_n^*$. Define
\begin{equation}\label{stab}
H^{\CO_{\udl{p}}}:=\Cent_{N_{G_n^*}(V_{\udl{p}})^\sm}(\psi_{\CO_{\udl{p}}})^\circ,
\end{equation}
the identity connected component of the centralizer.
As discussed in \cite[Section 2.3]{JZ-BF}, the $\psi_{\CO_{\udl{p}}}$-Fourier coefficient $\CF^{\psi_{\CO_{\udl{p}}}}(\varphi)(g)$ of $\varphi$ is left
$H^{\CO_{\udl{p}}}(F)$-invariant, smooth when restricted on $H^{\CO_{\udl{p}}}(\BA)$, and of moderate growth on a Siegel set of
$H^{\CO_{\udl{p}}}(\BA)$. For any $\pi\in\CA_\disc(G_n^*)$, we define $\CF^{{\CO_{\udl{p}}}}(\pi)$ to be the space spanned by all
$\CF^{\psi_{\CO_{\udl{p}}}}(\varphi_\pi)$ with $\varphi_\pi$ running in the space of $\pi$, and call $\CF^{{\CO_{\udl{p}}}}(\pi)$
the {\sl $\psi_{\CO_{\udl{p}}}$-Fourier module} of $\pi$.

For an $F$-pure inner form $G_n$ of $G_n^*$, a partition $\udl{p}$ in the set $\CP(\Fn)$ is called {\sl $G_n$-relevant} if
the unipotent subgroup $V_{\udl{p}}$ of $G_n$ as algebraic groups over the algebraic closure $\overline{F}$ is actually defined over $F$.
We denote by $\CP(\Fn)_{G_n}$ the subset of the set $\CP(\Fn)$ consisting of all $G_n$-relevant partitions of $\Fn$. It is easy to see that the above discussion about Fourier coefficients can be applied to all $\pi\in\CA_\disc(G_n)$ and
all $\udl{p}\in\CP(\Fn)_{G_n}$, without change.

It is clear that the regular nilpotent orbit in this case corresponds to the partition $\udl{p}_\reg=[\Fn]$ with $N=\Fn$, and the subregular nilpotent orbit
corresponds to the partition $\udl{p}_\subr=[(\Fn-1)1]$.

\section{A Conjecture on Global Arthur Packets}\label{sec-CGAP}

For $F$-quasisplit unitary groups $G_n^*$, it is known that any global Arthur packet $\wt{\Pi}_\phi(G_n^*)$ with a generic global Arthur parameter $\phi$
contains a generic member, i.e. contains a member that has a non-zero Fourier coefficient associated to the regular partition $\udl{p}_\reg=[\Fn]$.
This was proved by using the automorphic descent of Ginzburg-Rallis-Soudry, as explicitly given in \cite[Section 3.1]{JL-Cogdell}.

It is natural to ask what happens if the group is not $F$-quasisplit. We present here our conjecture for unitary groups. The conjecture
for other classical groups is similar and will be discussed in our future work.

Let $G_n=\RU(V,q)$ be the unitary group as defined in Section \ref{ssec-ug}, with $F$-rank $\Fr$. Recall that when $2\Fr<\Fn$, the root system
$_{F}\!\Phi(G_n,S)$ is non-reduced of type $BC_{\Fr}$; otherwise, it is of type $C_{\Fr}$. In this section, we assume that $2\Fr<\Fn$. This excludes
the case when $G_n=\RU_{n,n}$, which is $F$-quasisplit and whose root system is of type $C_\Fr$. In this case, the conjecture discussed below is already
known (\cite[Section 3.1]{JL-Cogdell}).

Assume that $2\Fr<\Fn$ and define the following partition of $\Fn$, depending on the $F$-rank $\Fr$ of $G_n$,
\begin{equation}\label{pFr}
\udl{p}_\Fr=[(2\Fr+1)1^{\Fn-2\Fr-1}].
\end{equation}
It is clear that the partition $\udl{p}_\Fr$ is $G_n$-relevant. Recall from Section 2.3 of \cite{JZ-BF}
that the partition $\udl{p}_\Fr$ defines the Fourier coefficients
$\CF^{\CO_{\udl{p}_\Fr}}$ of
Bessel type for automorphic forms on $G_n(\BA)$. In this case, following \eqref{stab}, we define
\begin{equation}\label{stab2}
H^{\CO_{\Fr}}:=\Cent_{N_{G_n}(V_{\udl{p}_\Fr})^\sm}(\psi_{\CO_{\udl{p}_\Fr}})^\circ.
\end{equation}
By Proposition 2.5 of \cite{JZ-BF}, $H^{\CO_{\Fr}}$ is a unitary group defined over $F$ and is a subgroup of the  $F$-anisotropic group $\RU(V_0,q)$.
For any $\pi\in\CA_\disc(G_n)$, as in \cite[Section 2.3]{JZ-BF},
we call the $\psi_{\CO_{\udl{p}}}$-Fourier module of $\pi$,
$$
\CF^{{\CO_{\udl{p}_\Fr}}}(\pi)=\CF^{{\CO_{\Fr}}}(\pi),
$$
the $\Fr$-th {\sl Bessel module} of $\pi$. Since $\Fr$ is the $F$-rank of $G_n$, by \cite[Proposition 2.2]{JZ-BF},
the $\Fr$-th Bessel module, if it is nonzero,
can be regarded as a submodule of the cuspidal spectrum $L^2_\cusp(H^{\CO_{\Fr}})$ in the space of all square integrable automorphic functions on
$H^{\CO_{\Fr}}(\BA)$. Note that the cuspidality here is not essential since the group $H^{\CO_{\Fr}}$ is $F$-anisotropic.

\begin{conj}\label{conj-main}
Let $G_n$ be the unitary group $F$-rank $\Fr$ with $2\Fr<\Fn$ and $G_n^*$ be the $F$-quasisplit pure inner form of $G_n$.
For any $\phi$ in the set $\wt{\Phi}_2(G_n^*)_{G_n}$ of
$G_n$-relevant generic global Arthur parameters, the global Arthur packet $\wt{\Pi}_\phi(G_n)$ contains a member $\pi_0$ that
has a nonzero $\Fr$-th Bessel module $\CF^{{\CO_{\Fr}}}(\pi_0)$ associated to the partition $\udl{p}_\Fr$ as defined in \eqref{pFr}.
\end{conj}

In the following, we are going to discuss the relation of Conjecture \ref{conj-main} with the non-vanishing of central value of certain $L$-functions.

Assume that $\pi\in\CA_\cusp(G_n)$ has a $G_n$-relevant, generic global Arthur parameter $\phi\in\wt{\Phi}_2(G_n^*)_{G_n}$. By Conjecture \ref{conj-main},
there exists a cuspidal member $\pi_0$ that shares with $\pi$ the global Arthur packet $\wt{\Pi}_\phi(G_n)$ and has the nonzero $\Fr$-th Bessel module
$\CF^{{\CO_{\Fr}}}(\pi_0)$. Following the discussion before Conjecture \ref{conj-main} and Proposition 2.2 of \cite{JZ-BF},
the nonzero $\Fr$-th Bessel module
$\CF^{{\CO_{\Fr}}}(\pi_0)$ can be regarded as a submodule of the cuspidal spectrum $L^2_\cusp(H^{\CO_\Fr})$, and
can be written, as a Hilbert direct sum
$$
\CF^{{\CO_{\Fr}}}(\pi_0)=\sigma_1\oplus\sigma_2\oplus\cdots\oplus\sigma_k\oplus\cdots
$$
of irreducible cuspidal automorphic representations of $H^{\CO_\Fr}(\BA)$. From the definition of the partition $\udl{p}_\Fr$ in \eqref{pFr},
by the {\sl Generic Summand Conjecture}, which is Conjecture 2.3 of \cite{JZ-BF},
there exists an irreducible cuspidal automorphic representation $\sigma$ of $H^{\CO_{\Fr}}(\BA)$ with
a generic global Arthur parameter $\phi_\sigma$, such that the $L^2$-inner product
\begin{equation}\label{conj23}
\left<\CF^{\psi_{\CO_{\Fr}}}(\varphi_{\pi_0}),\varphi_\sigma\right>_{H^{\CO_{\Fr}}}
\end{equation}
is nonzero for some $\varphi_{\pi_0}\in\pi_0$ and $\varphi_\sigma\in\sigma$. Then, by Theorem 5.5 of \cite{JZ-BF}, which is one direction of the
global Gan-Gross-Prasad conjecture (\cite{GGP12}), we obtain that the central value $L(\frac{1}{2},\pi_0\times\sigma)$ is nonzero. Here the
tensor product $L$-function is defined in terms of the global Arthur parameters $\phi$ and $\phi_\sigma$ by
$$
L(s,\phi\times\phi_\sigma)=L(s,\pi_0\times\sigma),
$$
following \cite{A13}, \cite{Mk15}, and \cite{KMSW}. Since $\pi$ and $\pi_0$
share the same generic global Arthur packet $\wt{\Pi}_\phi(G_n)$, we obtain that
$$
L(\frac{1}{2},\pi\times\sigma)=L(\frac{1}{2},\pi_0\times\sigma)\neq 0.
$$

\begin{thm}\label{th-nva}
Assume that Conjecture 2.3 of \cite{JZ-BF} holds. Then Conjecture \ref{conj-main} implies that
for any $\pi\in\CA_\cusp(G_n)$
with a $G_n$-relevant, generic global Arthur parameter $\phi\in\wt{\Phi}_2(G_n^*)_{G_n}$, there exists a $\sigma\in\CA_\cusp(H^{\CO_{\Fr}})$, which depends only on $\phi$ and has
a generic global Arthur parameter $\phi_\sigma$, such that the central value $L(\frac{1}{2},\pi\times\sigma)$ is nonzero.
\end{thm}

If one writes out explicitly the generic global Arthur parameters:
$$
\phi=(\tau_1,1)\boxplus\cdots\boxplus(\tau_r,1)
$$
and
$$
\phi_\sigma=(\tau_1',1)\boxplus\cdots\boxplus(\tau_{r'}',1),
$$
then we have the following simultaneous non-vanishing result.

\begin{cor}
Assume that Conjecture 2.3 of \cite{JZ-BF} holds. Then Conjecture \ref{conj-main} implies that
with the notation above, the following simultaneous non-vanishing property holds
$$
L(\frac{1}{2},\phi\times\phi_\sigma)=
\prod_{i=1}^r\prod_{j=1}^{r'}L(\frac{1}{2},\tau_i\times\tau_j')\neq 0.
$$
\end{cor}

\section{Non-vanishing of the Central Value: $\RU_\Fn\times\RU_1$ Case}\label{sec-unu1}

Based on Conjecture \ref{conj-subr}, which is a special case of Conjecture \ref{conj-main}, we are able to establish the non-vanishing property of the central value of the tensor product $L$-functions
$L(\frac{1}{2},\pi\times\chi)$ for any irreducible cuspidal automorphic representation $\pi$ of the $F$-quasisplit unitary group $G_n^*(\BA)$ with
an automorphic character $\chi$ of $U_1(\BA)$. We treat the case of even $\Fn$ and the case of odd $\Fn$ separately.
We note that Conjecture 2.3 of \cite{JZ-BF} is known for the subregular partition (Proposition 7.4 of \cite{JZ-BF}).
In this case, we have that $\sigma=\chi$.
Hence the results in this section only need the assumption that Conjecture \ref{conj-subr} holds.

\subsection{The case of $\RU_{n,n}$}\label{ssec-Fneven}
We consider $G_n^*=\RU_{n,n}$ with $\Fn=2n$ and $\Fr=n$, which is $F$-quasisplit, and first prove the following proposition.

\begin{prop}\label{prop-2n}
For any generic global Arthur parameter $\phi$ of $\RU_{n,n}$, there exists a generic global Arthur parameter $\phi_2$ of $\RU_{1,1}$ that is
$\RU_{2,0}$-relevant, for some $\RU_{2,0}\in\CU_{2,0}$, such that the global Arthur parameter
$$
\phi\boxplus\phi_2
$$
is a generic global Arthur parameter of $\RU_{n+1,n+1}$ that is $\RU_{n+2,n}$-relevant, for some $\RU_{n+2,n}\in\CU_{n+2,n}$.
\end{prop}

\begin{proof}
Let  $\xi_{\chi_\kappa}$ be an $L$-embedding of ${}^L\RU_{n,n}$ to ${}^LG_{E/F}(2n)$, as recalled in Section \ref{ssec-dsap}.
Other unexplained notation can be found in Chapter 2 of \cite{Mk15} and in \cite{GGP12}.
The given generic global Arthur parameter $\phi$ belongs to $\wt{\Phi}_2(\RU_{n,n},\xi_{\chi_\kappa})$.
Then, as a parameter in  $\wt{\Phi}_\el(2n)$, $\phi$, which is denoted by $\phi^{2n}$ in \cite{Mk15} and also as in \eqref{ellggap},
is conjugate self-dual of sign $-\kappa$.
This means that the Asai $L$-function $L(s,\phi,\Asai^{-\kappa})$ has a pole at $s=1$.

Take $\phi_2$ in $\wt{\Phi}_{\rm sim}(\RU_{1,1},\xi_{\chi_\kappa})$, and define
$\phi':=\phi\boxplus \phi_2$.
By \cite[Section 2.4]{Mk15}, $\phi'$ belongs to $\wt{\Phi}_2(\RU_{n+1,n+1},\xi_{\chi_\kappa})$ and $\CZ:=\CZ(\GL_{2n+2}(\BC))^{\Gamma_{E/F}}$
arises a non-trivial subgroup of the component group $\CS_{\phi'}$.
Equivalently, under the $L$-embedding $\xi_{\chi_\kappa}$, as a parameter in $\wt{\Phi}_\el(2)$, $\phi_2$ is conjugate self-dual of the same sign
with $\phi$.

Let $\phi_v$ be the localization of $\phi$ at a local place  $v$.
Fix two finite local places $v_i$ of $F$, with $i\in\{1,2\}$, and assume that the two finite local places satisfy the following three conditions:
\begin{enumerate}
	\item \label{item:v-1} $v_1$ and $v_2$ are inert;
	\item \label{item:v-2} for any $i\in\{1,2\}$, the characteristic of the residue field of $F_{v_i}$ is not 2; and
	\item \label{item:v-3} $\phi_{v_1}$ and $\phi_{v_2}$ are unramified.
\end{enumerate}
Choose an infinite dimensional irreducible automorphic representation $\pi''$ of  $\RU_{2,0}(\BA_F)$ (where $\RU_{2,0}$ is an $F$-anisotropic form of $\RU_{1,1}$) such that $\pi''_{v_1}$ and $\pi''_{v_2}$ are  supercuspidal.
The existence of  $\pi''$ is given by a simple version of the Arthur-Selberg trace formula (refer to  \cite[Theorem 3.1]{G96}, for instance).
Take $\phi_2$ to be the global Arthur parameter of $\pi''$. It follows that $\phi_2$ belongs to $\wt{\Phi}_{\rm sim}(\RU_{1,1},\xi_{\chi_\kappa})$, and
then $\phi'$ belongs to $\wt{\Phi}_2(\RU_{n+1,n+1},\xi_{\chi_\kappa})$.

It remains to show that $\phi'$ is $\RU_{n+2,n}$-relevant for some $\RU_{n+2,n}\in\CU_{n+2,n}$.
It is enough to construct a member in the global Arthur packet $\wt{\Pi}_{\phi'}(\RU_{n+2,n})$ for a certain $\RU_{n+2,n}\in\CU_{n+2,n}$.

Since $v_i$ is inert, as a parameter for $G_{E_{v_i}/F_{v_i}}(2n+2)$, $\phi'_{v_i}$ can be written as
\begin{equation}\label{eq:A-unram}
\oplus_{j}(\chi_{i,j}\oplus \iota(\chi_{i,j})^{-1})\oplus n_i\cdot 1\oplus m_i\cdot \xi_{E_{v_i}/F_{v_i}}\oplus \phi_{2,v_i},	
\end{equation}
where $\xi_{E_{v_i}/F_{v_i}}$ is the nontrivial unramified quadratic character of $E_{v_i}$,
$m_i$ and $n_i$ are the multiplicities of the characters $1$ and $\xi_{E_{v_i}/F_{v_i}}$ respectively, and $\chi_{i,j}\ne \iota(\chi_{i,j})^{-1}$.

Note that $\xi_{E_{v_i}/F_{v_i}}\vert_{F^\times_{v_i}}$ equals the quadratic character $\omega_{E_{v_i}/F_{v_i}}$ of $F^\times_{v_i}$ by the local class field theory. It follows that $\xi_{E_{v_i}/F_{v_i}}$ is conjugate symplectic.
Since $\pi''_{v_i}$ is supercuspidal, the corresponding local Arthur parameter $\phi_{2,v_i}$ is simple, and $\CS_{\phi'_{v_i}}$ is non-trivial.
We may regard $\CZ$ as a non-trivial subgroup of $\CS_{\phi'_{v_i}}$.
Following from \cite[Chapter 1]{KMSW} (or  \cite[Section 25]{GGP12}),
the local Vogan packet $\wt{\Pi}_{\phi'_{v_i}}[G^*_{n+1}]$ contains at least two members.
Let $\chi_{v_i,-}$ be the character of $\CS_{\phi'_{v_i}}$
with the property that the restriction of $\chi_{v_i,-}$ into the component group of the centralizer of $\phi_{2,v_i}$ is non-trivial and the restriction into the other component  is trivial.
Via the local Langlands correspondence as established in \cite{KMSW}, take $\pi'_{v_i,-}$ to be the representation in $\wt{\Pi}_{\phi'_{v_i}}[G^*_{n+1}]$
corresponding to the character $\chi_{v_i,-}$.
Since $\chi_{v_i,-}\vert_{\CZ}$ is non-trivial,
$\pi'_{v_i,-}$ is defined on the unique non-quasisplit unitary group $\RU_{n+2,n}(F_{v_i})$ of $F$-rank $n$.

For local places $v\ne v_i$, take $\pi'_{v}$ to be the generic member in $\wt{\Pi}_{\phi'_{v}}[G^*_{n+1}]$ whose associated character $\chi_v$ of $\CS_{\phi'_v}$ is trivial.
Then we take the restricted tensor product with the chosen local representations:
$$
\pi'=(\otimes_{v\ne v_i}\pi'_v)\otimes\pi'_{v_1,-}\otimes \pi'_{v_2,-}.
$$
It is clear that $\pi'$ is a member in the global Vogan packet $\wt{\Pi}_{\phi'}[G^*_{n+1}]$.
By  \cite[Proposition 1.2.3]{HL04}, there exists a global pure inner form $G_{n+1}$ of $G_{n+1}^*$ such that $G_{n+1}(F_{v})$ is quasisplit when $v\ne v_i$
and non-quasisplit when $v=v_i$ for $i=1,2$.
By the Hasse principle for Hermitian spaces, $G_{n+1}$ is of $F$-rank $n$, and is an $F$-pure inner form of $\RU_{n+1,n+1}$ of type $\RU_{n+2,n}$, that is,
$G_{n+1}\in\CU_{n+2,n}$.
Hence $\pi'$ is coherent and defined on  the unitary group $G_{n+1}$.
Following from \cite[Chapter 1]{KMSW} (or  \cite[Section 25]{GGP12}), $\pi'$ is an irreducible cuspidal automorphic representation of $G_{n+1}(\BA)$ with multiplicity one, because
$$
\prod_{v\ne v_i}\chi_{v}\cdot \chi_{v_1,-}\cdot \chi_{v_2,-}=1
$$
is the trivial character of $\CS_{\phi'}$. Because $\pi'$ has $\phi'$ as its global Arthur parameter, we obtain that
$\phi'$ is $\RU_{n+2,n}$-relevant for some $\RU_{n+2,n}\in\CU_{n+2,n}$. We are done.
\end{proof}

Now we can establish the following non-vanishing result.

\begin{thm}\label{th-nveven}
Assume that Conjecture \ref{conj-subr} holds for unitary groups $\RU_{n+2,n}\in\CU_{n+2,n}$.
For any generic global Arthur parameter $\phi$ of $\RU_{n,n}$, which is $F$-quasisplit,
there exists an automorphic character $\chi$ of $\RU_1(\BA)$ such that the
central value of the tensor product $L$-function $L(\frac{1}{2},\pi\times\chi)$ is nonzero for all automorphic members $\pi$ in the global Arthur packet
$\wt{\Pi}_\phi(\RU_{n,n})$.
\end{thm}

\begin{proof}
For any generic global Arthur parameter $\phi$ of $\RU_{n,n}$, by Proposition \ref{prop-2n}, there exists
a generic global Arthur parameter $\phi_2$ of $\RU_{1,1}$ that is
$\RU_{2,0}$-relevant for some $\RU_{2,0}\in\CU_{2,0}$, such that the endoscopic sum
$
\phi\boxplus\phi_2
$
is a generic global Arthur parameter of $\RU_{n+1,n+1}$ that is $\RU_{n+2,n}$-relevant for some $\RU_{n+2,n}\in\CU_{n+2,n}$.

By applying Conjecture \ref{conj-subr} to the generic global Arthur parameter $\phi\boxplus\phi_2$ and the unitary group $\RU_{n+2,n}$,
we deduce that there exists an automorphic member $\Pi_0$ in the global Arthur packet $\wt{\Pi}_{\phi\boxplus\phi_2}(\RU_{n+2,n})$ with the property that
$\Pi_0$ has a nonzero Fourier coefficient associated to the subregular partition $\udl{p}_\subr=\udl{p}_{n}=[(2n+1)1]$.

We note that for $\udl{p}_{\subr}=[(2n+1)1]$ of $2n+2$, the subgroup $H^{\CO_n}$ is $\RU_1$, and hence Conjecture 2.3 of \cite{JZ-BF}
holds. We refer to Proposition 7.4 of \cite{JZ-BF} for the detail.
Therefore, by applying Theorem \ref{th-nva} to the current situation, we obtain that there exists an automorphic character $\chi$ of
$\RU_1(\BA)$ such that the central value of the tensor product $L$-function $L(\frac{1}{2},\Pi_0\times\chi)$ is nonzero.

Now for any automorphic member $\pi$ of the global Arthur packet $\wt{\Pi}_\phi(\RU_{n,n})$ and any automorphic member $\sigma$ of the global Arthur packet
$\wt{\Pi}_{\phi_2}(U_{2,0})$, the endoscopic transfer of $\pi$ and $\sigma$, which is denoted by $\pi\boxplus\sigma$, belongs to the global
Arthur packet $\wt{\Pi}_{\phi\boxplus\phi_2}(\RU_{n+2,n})$, following \cite{KMSW}. Hence we obtain the following identity
$$
L(\frac{1}{2},\Pi_0\times\chi)=L(\frac{1}{2},(\pi\boxplus\sigma)\times\chi)=L(\frac{1}{2},\pi\times\chi)\cdot L(\frac{1}{2},\sigma\times\chi).
$$
In particular, the central $L$-value $L(\frac{1}{2},\pi\times\chi)$ is nonzero. We are done.
\end{proof}

\subsection{The case of $\RU_{n+1,n}$}\label{ssec-Fnodd}
We consider $G_n^*=\RU_{n+1,n}$ when $\Fn=2n+1$ and $\Fr=n$, which is $F$-quasisplit, and prove the following analogue of Proposition \ref{prop-2n}.

\begin{prop}\label{prop-2n1}
For any generic global Arthur parameter $\phi$ of $\RU_{n+1,n}$, there exists a  generic global Arthur parameter $\phi_1$ of $\RU_{1}$,
such that the global Arthur parameter
$$
\phi\boxplus\phi_1
$$
is a generic global Arthur parameter of $\RU_{n+1,n+1}$ that is $\RU_{n+2,n}$-relevant for some $\RU_{n+2,n}\in\CU_{n+2,n}$.
\end{prop}

\begin{proof}
The proof is similar to that of Proposition \ref{prop-2n}.

Suppose that $\phi$ is in $\wt{\Phi}_2(G^*_n,\xi_\chi)$.
Take a $\phi_1$ in $\wt{\Phi}_{\rm sim}(\RU_1)$ and define $\phi':=\phi\boxplus \phi_1$.
Then $\phi'$ is in $\wt{\Phi}_2(G^*_{n+1},\xi_{\chi_{-\kappa}})$. The key point is to show that $\phi'$ is $\RU_{n+2,n}$-relevant for
some $\RU_{n+2,n}\in\CU_{n+2,n}$.

Note that under the $L$-embedding $\xi_{\chi_{-\kappa}}$, $\phi_1$ is conjugate self-dual of the same sign with $\phi$, as the parameters for
$G_{E/F}(1)$ and $G_{E/F}(2n+1)$ respectively.

Similarly, fix two finite local places $v_i$ of $F$ with $i\in\{1,2\}$ and assume that they have the properties
\eqref{item:v-1}, \eqref{item:v-2} and \eqref{item:v-3} as in the proof of Proposition \ref{prop-2n}.
Since $\phi_1$ and $\phi$ are conjugate self-dual of the same sign,
we may choose  $\phi_1\in\wt{\Phi}_{\rm sim}(\RU_1)$ such that at local places $v_i$,
$\phi_{1,v_i}=1$ if $\phi_{v_i}$ is conjugate orthogonal and $\phi_{1,v_i}=\xi_{E_{v_i}/F_{v_i}}$ if $\phi_{v_i}$ is conjugate-symplectic.
Write $\phi_{v_i}$ as
$$
\oplus_{j}(\chi_{i,j}\oplus \iota(\chi_{i,j})^{-1})\oplus n_i\cdot 1\oplus m_i \cdot \xi_{E_{v_i}/F_{v_i}}.	
$$
Here the notation is the same with \eqref{eq:A-unram}.
By definition of $\phi_1$,
$\phi'_{v_i}$ have summand $$(n_i+1)\cdot 1\oplus m_i\cdot \xi_{E_{v_i}/F_{v_i}}$$ with even $m_i$ if $\phi_{v_i}$ is conjugate orthogonal;
and $$n_i\cdot 1\oplus (m_i+1)\cdot \xi_{E_{v_i}/F_{v_i}}$$ with even $n_i$ if $\phi_{v_i}$ is conjugate symplectic.
Following from Proposition 21.1 (iii) and (iv) of \cite{GGP12}, we obtain that the local component group $\CS_{\phi'_{v_i}}=\CZ\cong\BZ_2$.
Hence the local Vogan packet $\wt{\Pi}_{\phi'_{v_i}}[G^*_{n+1}]$ contains two members.
Write
$$
\wt{\Pi}_{\phi'_{v_i}}[G^*_{n+1}]=\{\pi'_{v_i,+},\pi'_{v_i,-}\}
$$
where $\pi'_{v_i,-}$ (resp.\ $\pi'_{v_i,+}$) corresponds to the non-trivial (resp.\ trivial) character of $\CS_{\phi'_{v_i}}\cong\BZ_2$.
Moreover, $\pi'_{v_i,-}$ is defined over the unique non-quasisplit unitary group $\RU_{n+2,n}(F_{v_i})$ of $F$-rank $n$.

Take a restricted tensor product
$$
\pi'=(\otimes_{v\ne v_i}\pi_v)\otimes\pi'_{v_1,-}\otimes\pi'_{v_2,-}
$$
in the global Vogan packet $\wt{\Pi}_{\phi'}[\RU_{n+1,n+1}]$, where $\pi'_v$ for $v\ne v_i$ is the local generic representation corresponding to the trivial character of $\CS_{\phi'_v}$.
Similarly, applying \cite[ Proposition 1.2.3]{HL04} and  \cite[Chapter 1]{KMSW}, we know that $\pi'$ is an irreducible cuspidal automorphic representation of $\RU_{n+2,n}(\BA)$ for some $\RU_{n+2,n}\in\CU_{n+2,n}$, and has $\phi'$ as its global Arthur parameter. Therefore, the global Arthur parameter $\phi'$ is $\RU_{n+2,n}$-relevant for some $\RU_{n+2,n}\in\CU_{n+2,n}$. We are done.
\end{proof}

Similarly, we can establish the following non-vanishing result.

\begin{thm}\label{th-nvodd}
Assume that Conjecture \ref{conj-subr} holds for unitary groups $\RU_{n+2,n}\in\CU_{n+2,n}$.
For any generic global Arthur parameter $\phi$ of $\RU_{n+1,n}$, which is $F$-quasisplit, there exists an automorphic character $\chi$ of $\RU_1(\BA)$ such that the
central value of the tensor product $L$-function $L(\frac{1}{2},\pi\times\chi)$ is nonzero for all automorphic members $\pi\in\wt{\Pi}_\phi(\RU_{n+1,n})$.
\end{thm}

The proof is of the same nature as that of Theorem \ref{th-nveven}, and will be omitted here.

\subsection{Some more non-vanishing results}\label{ssec-mnv}
Recall that the generic global Arthur parameters of $\RU^*_{E/F}(\Fn)$ are of the form
$$
\phi_\Fn=(\tau_1,1)\boxplus\cdots\boxplus(\tau_r,1)
$$
as in \eqref{ellggap}, where for $i=1,2,\cdots,r$, $\tau_i$ is an irreducible
cuspidal automorphic representation of $G_{E/F}(a_i)$ that is conjugate self-dual.
The following non-vanishing property follows from Theorems \ref{th-nveven} and \ref{th-nvodd}.

\begin{thm}\label{th-nvgl}
Let $F$ be a number field, and $E$ be a quadratic extension of $F$.
For any positive integers $a_1,a_2,\cdots,a_r$, let $\tau_i$ be an irreducible cuspidal automorphic representation of $G_{E/F}(a_i,\BA)=\GL_{a_i}(\BA_E)$.
Assume that the following hold:
\begin{enumerate}
\item All the $\tau_i$'s are conjugate self-dual, so that $\phi_\Fn=(\tau_1,1)\boxplus\cdots\boxplus(\tau_r,1)$ with $\Fn=\sum_{i=1}^ra_i$ is
a generic global Arthur parameter of the $F$-quasisplit unitary group $\RU^*_{E/F}(\Fn)$; and
\item Conjecture \ref{conj-subr} holds for unitary groups $\RU_{n+2,n}\in\CU_{n+2,n}$ with $n=[\frac{\Fn}{2}]$.
\end{enumerate}
Then there exists an automorphic character $\chi$ of $\RU_{E/F}(1)$ such that the following product is nonzero:
\begin{equation}
L(\frac{1}{2},\tau_1\times\chi)\cdot L(\frac{1}{2},\tau_2\times\chi)\cdots L(\frac{1}{2},\tau_r\times\chi)\neq 0.
\end{equation}
\end{thm}

As described in Section \ref{ssec-dsap}, the first assumption for the generic global Arthur parameters is easy to verify if one is given a
pair of generic global Arthur parameters of unitary groups. We are going to prove Theorem \ref{th-nvgl} for $\Fn\leq 4$ without the assumption of
Conjecture \ref{conj-subr} in Section \ref{sec-n4}.
Theorem \ref{th-nvgl} can be reformulated in terms of endoscopic classification for unitary groups as follows.

\begin{thm}\label{th-nvug}
Let $\phi_\Fn$ be a generic global Arthur parameter of the $F$-quasisplit unitary group $\RU^*_{E/F}(\Fn)$, and let $\phi_\Fm$ be
a generic global Arthur parameter of the $F$-quasisplit unitary group $\RU^*_{E/F}(\Fm)$. Assume that $\phi_\Fn\boxplus\phi_\Fm$ is
a generic global Arthur parameter of the $F$-quasisplit unitary group $\RU^*_{E/F}(\Fn+\Fm)$, and assume that Conjecture \ref{conj-subr}
holds for unitary groups $\RU_{r+2,r}\in\CU_{r+2,r}$ with $r=[\frac{\Fn+\Fm}{2}]$.
Then there exists an automorphic character $\chi$ of $\RU_{E/F}(1)$ such that
$$
L(\frac{1}{2},\pi\times\chi)\cdot L(\frac{1}{2},\sigma\times\chi)\neq 0
$$
for any automorphic member $\pi$ in the global Arthur packet $\wt{\Pi}_{\phi_\Fn}(\RU^*_{E/F}(\Fn))$ and
any automorphic member $\sigma$ in $\wt{\Pi}_{\phi_\Fm}(\RU^*_{E/F}(\Fm))$.
\end{thm}

We are also going to prove Theorem \ref{th-nvug} for $\Fn+\Fm\leq 4$ without the assumption on Conjecture \ref{conj-subr} in Section \ref{sec-n4}.

\section{Non-vanishing Property for $\RU_{2,2}\times \RU_1$}\label{sec-n4}

We are going to prove Theorem \ref{th-U22} with an approach different from those in the previous sections. It is to work directly with $F$-quasisplit unitary
groups. However, for the moment, we only work out for $\Fn\leq 4$.
We first consider the case of $\RU_{2,2}$ and then deduce the other cases by the
theory of endoscopy.
We set the $F$-quasisplit unitary group $G_2^*=\RU_{2,2}=\RU^*_{E/F}(4)=\RU(V^*,q^*)$, with
$$
V^*=\Span\{e_1,e_2\}\oplus \Span\{e_{-2},e_{-1}\}
$$
where $q^*(e_i,e_{-j})=\delta_{ij}$. As proved in \cite[Section 3.1]{JL-Cogdell}, each generic global Arthur packet contains a generic automorphic member.
It is enough to show

\begin{thm}\label{thm:main-U}
For any irreducible generic cuspidal automorphic representation $\pi$ of $\RU_{2,2}(\BA)$, there exists an automorphic character $\chi$ of $\RU_1(\BA)$ such that
$$
L(\frac{1}{2},\pi\times\chi)\ne 0.
$$
\end{thm}

As indicated in the discussion of Section \ref{sec-unu1}, it is enough to show that any irreducible generic cuspidal automorphic representation $\pi$ of
$\RU_{2,2}(\BA)$ has a non-zero Fourier coefficient associated to the subregular partition $\udl{p}_{\subr}=[31]$. In this situation, the stabilizer
$H^{\CO_{\udl{p}_\subr}}$ is $\RU_{E/F}(1)=\RU_1$, and hence Conjecture 2.3 of \cite{JZ-BF} holds (Proposition 7.4 of \cite{JZ-BF}). Therefore, the following
proposition implies Theorem \ref{thm:main-U}.

\begin{prop}\label{prop-u4subr}
Let $\pi$ be an irreducible cuspidal automorphic representation of $\RU_{2,2}(\BA)$. Assume that $\pi$ is generic, i.e.
it has a non-zero Whittaker-Fourier coefficient. Then $\pi$ has a nonzero Fourier coefficient associated to the
subregular partition $\udl{p}_{\subr}=[31]$.
\end{prop}

The proof of Proposition \ref{prop-u4subr} takes place in the following section. Before we go to there, we would like to note that the following
conjectural generalization of Proposition \ref{prop-u4subr} should be expected.
If this is the case, then we are able to extend Theorem \ref{thm:main-U} to general
quasi-split unitary group $\RU_{n,n}$.

\begin{conj}\label{conj-subr2n}
Let $\pi$ be an irreducible cuspidal automorphic representation of $\RU_{n,n}(\BA)$. Assume that $\pi$ is generic, i.e.
it has a non-zero Whittaker-Fourier coefficient. Then $\pi$ has a nonzero Fourier coefficient associated to the
subregular partition $$\udl{p}_{\subr}=[(2n-1)1].$$
\end{conj}

However, due to the limitation of the current understanding of the $F$-rational structure of the Fourier coefficients of automorphic representations,
we are only able to prove the case when $n=2$ for this moment, which is given below.

\subsection{Proof of Proposition \ref{prop-u4subr}}\label{sec-proofu4}
For the group $\RU_{2,2}$, all stable nilpotent orbits are parametrized by the following partitions $\underline{p}$,
\begin{equation}\label{eq:p-4}
[4],\qquad [31],\qquad[2^2],\qquad [21^2],\qquad [1^4],
\end{equation}
respectively, where the power indicates the multiplicity of the parts in the partitions.
Remark that those orbits are listed in the topological order. According to \cite{Li92}, cuspidal automorphic forms are nonsingular. Hence
every cuspidal automorphic representation has a nonzero Fourier coefficient associated to the partition $[2^2]$, and we only consider the
following three partitions
$$
[4],\qquad [31],\qquad[2^2].
$$

Recall the notation in Section \ref{ssec-ug}.
Let $P_0=TN_0$ be the standard Borel subgroup  (the fixed minimal parabolic subgroup) of $\RU_{2,2}$ consisting of all upper triangular matrices,
where $T:=\Res_{E/F}S$ for short.
Let $P_{\hat{i}}=M_{\hat{i}}N_{\hat{i}}$ with $i\in \{1,2\}$ be the standard maximal parabolic subgroup of $\RU_{2,2}$
whose Levi subgroup $M_{\hat{i}}$ is isomorphic to $$G_{E/F}(i)\times \RU_{E/F}(4-2i).$$
Then $N_{\hat{i}}$ is the unipotent subgroup of the nilpotent orbit associated with the partition $[(4-i)i]$.

For a nilpotent orbit $\CO_{\underline{p}}$ associated with the partition $\underline{p}=[31]$, we have that $V_{\CO_{\underline{p}}}=N_{\hat{1}}$. It is easy to check that
$V_{\CO_{\underline{p}}}(F)/[V_{\CO_{\underline{p}}}(F),V_{\CO_{\underline{p}}}(F)]$ is isomorphic to
\begin{equation}\label{eq:p-1}
\Span_E\{e_{2},e_{-2}\}.	
\end{equation}
We still denote by $q^*$ the restriction of $q^*$ to this subspace.
Fix an additive character $\psi_F$ of $F\bks \BA$.
For each vector $v$ in \eqref{eq:p-1}, define the character $\psi_{\underline{p};v}$ by
\begin{equation}\label{eq:psi-3-1}
\psi_{[31];v}(\begin{pmatrix}
1&x&x''\\
&I&x'\\
&&1	
\end{pmatrix})=\psi_{F}(\frac{1}{2}\tr_{E/F}q^*(x,v)),	
\end{equation}
where $x'$ is determined by $x$, as the unipotent elements are in $\RU_{2,2}$.
For example,
for $\alpha\in E$, take
\begin{equation}
v_\alpha=e_2+\alpha e_{-2}.
\end{equation}
The character $\psi_{[31];v_\alpha}$ of $N_{\hat{1}}(F)\bks N_{\hat{1}}(\BA)$ is given by
\begin{align*}
 \psi_{[31];v_\alpha}(\begin{pmatrix}
1&u&x_{1,1}&x_{1,2}\\
&1&0&-\iota(x)_{1,1}\\
&&1&-\iota(u)\\
&&&1	
\end{pmatrix})= &\psi_F(\frac{1}{2}\tr_{E/F}(q^*(ue_2+x_{1,1}e_{-2},e_2+\alpha e_{-2})))\\
=&\psi_F(\frac{1}{2}\tr_{E/F}(\iota(\alpha) u+x_{1,1})),
\end{align*}
where
the entries $u$, $x_{1,1}$ and $x_{1,2}$ are in $\BA_E$.
By
\begin{equation}\label{eq:length-v}
q^*(v_{\alpha},v_{\alpha})=\alpha+\iota(\alpha),
\end{equation}
$v_{\alpha}$ is anisotropic when $\tr_{E/F}\alpha\ne 0$.
By the adjoint action of $M_{\hat{1}}(F)\cong \GL_1(E)\times \RU_{1,1}(F)$ on $\psi_{[31];v_\alpha}$,
its stabilizer $M_{[31];v_\alpha}(F)$ is isomorphic to $\RU_1(F)$ if $\tr_{E/F}\alpha\ne 0$,
and isomorphic to $\GL_1(E)\rtimes E$ if $\tr_{E/F}\alpha=0$.
By definition, the $[31]$-type Fourier coefficient $\CF_{[31];v_\alpha}$ is given by such an $\alpha$ that $\tr_{E/F}\alpha\ne 0$.

For a nilpotent orbit $\CO_{\underline{p}}$ associated with the partition $\underline{p}=[2^2]$, the dual group of
$V_{\CO_{\underline{p}}}(F)/[V_{\CO_{\underline{p}}}(F),V_{\CO_{\underline{p}}}(F)]$ is isomorphic to
\begin{equation}\label{eq:p-2-2}
\{X\in M_{2\times 2}(E)\colon \iota(X)+w_2X^tw_2=0\},
\end{equation}
where $w_2$ is the $2\times 2$ anti-diagonal matrix, i.e., $w_2=J^2_2$ in \eqref{eq:J}.
For each $X$ in the set \eqref{eq:p-2-2},
define the character
\begin{equation}\label{eq:psi-2-2}
\psi_{\underline{p};X}(n)=\psi_{F}(\frac{1}{2}\tr(XA))	
\end{equation}
where $n\in N_{\hat{2}}$ is of form
$\ppair{\begin{smallmatrix}
I_2&A\\ 0&I_2
\end{smallmatrix}}$.
Note that $\tr(XA)=\iota(\tr(XA))$ and all characters of $V_{\CO_{\underline{p}}}(\BA)$,
which are trivial on $V_{\CO_{\underline{p}}}(F)$,  are defined this way.

With the preparation above, we are ready to prove Proposition \ref{prop-u4subr}. To avoid the confusion of notation,
we use $\BA=\BA_F$ for the ring of adeles of $F$, and $\BA_E$ for the ring of adeles of $E$. For any closed algebraic subgroup $X$ of $\RU_{2,2}$,
$X(\BA)$ is the $\BA$-points of $X$. However, when the elements of $X$ is written as matrices for the calculation, we will specify if
the matrix elements belong to $\BA_F$ or $\BA_E$. Note that the explicit calculation may involve certain non-algebraic subgroup of $\RU_{2,2}$. In such a
circumstance, we have to make the matrix elements of such a subgroup belong to $\BA_F$ or $\BA_E$.

The idea to prove Proposition \ref{prop-u4subr} can be briefly outlined as follows. By \cite{Li92}, every cuspidal automorphic representation $\pi$ of
$\RU_{2,2}(\BA)$ has a non-zero Fourier coefficients associated with the partition $[2^2]$. The first point here is to show that if $\pi$ is generic,
i.e. has a non-zero Fourier coefficient associated with the partition $[4]$ (the Whittaker-Fourier coefficient), then $\pi$ must have
a non-zero Fourier coefficients associated to the partition $[2^2]$ with $X=\vartheta:=\diag(1,-1)$ (Lemma \ref{lm:U-22}). Then the key structure to
make our current argument work for $\RU_{2,2}$ is that the stabilizer of the character $\psi_{[2^2];\vartheta}$ is isomorphic to $\RU_{1,1}$, which leads to
Lemma \ref{lm:2comp22} with a further Fourier expansion. The last step is Lemma \ref{lm:comp-31} that determines structure of the composition of two
consecutive Fourier coefficients, which is a very technical issue in general. However, we resolve this issue for the current case by reducing it
to \cite[Corollary 7.1]{GRS11}. This proves Proposition \ref{prop-u4subr}.

First, we investigate the Fourier coefficients associated with the partition $[2^2]$.
Take $X=\vartheta=\diag(1,-1)$.
One can define the character $\psi_{[2^2];\vartheta}$ as in \eqref{eq:psi-2-2}, i.e.,
$$
\psi_{[2^2];\vartheta}(\begin{pmatrix}
1&0&x_{1,1}+\sqrt{\varsigma}y_{1,1}&\sqrt{\varsigma}y_{1,2}\\
&1&\sqrt{\varsigma}y_{2,1}&-x_{1,1}+\sqrt{\varsigma}y_{1,1}\\
&&1&0\\
&&&1	
\end{pmatrix} )=\psi_F(x_{1,1}),
$$
where $x_{1,1}$ and $y_{i,j}$ are in $\BA_F$.
Then the stabilizer $M_{\underline{p};\vartheta}$ is
$$
\cpair{
	m(g):=\begin{pmatrix}
		g& \\ &\wh{\iota(g)}^{-1}
	\end{pmatrix}\in G^*_2\colon
	g^t  \begin{pmatrix}
	 &1 \\ -1&	
	\end{pmatrix}
	\iota(g)=\begin{pmatrix}
	 &1 \\ -1&	
	\end{pmatrix}
},
$$
where $\hat{g}:=w_2g^{t}w_2$.
We may identify $M_{[2^2];\vartheta}$ as $\RU_{1,1}$,
preserving the  Hermitian matrix
$$
\begin{pmatrix}
&\sqrt{\varsigma}\\ -\sqrt{\varsigma}&	
\end{pmatrix}.
$$
The Lie algebra of $M_{[2^2];\vartheta}$ consists of elements
\begin{equation}\label{eq:U(1,1)}
\begin{pmatrix}
a+\sqrt{\varsigma}d&b\\ c&-a+\sqrt{\varsigma}d	
\end{pmatrix}.
\end{equation}
Remark that in general, for $X$ in \eqref{eq:p-2-2} with $\det(X)\ne 0$, the stabilizer $M_{[2^2];\vartheta}$ is isomorphic to a
pure inner form of  $\RU_{1,1}$.

\begin{lem}\label{lm:U-22}
Any irreducible generic cuspidal automorphic representation $\pi$ of $\RU_{2,2}(\BA)$ has a nonzero Fourier coefficient
$\CF^{\psi_{[2^2];\vartheta}}(\pi)$.
\end{lem}

\begin{proof}
Since $\pi$ is generic, any automorphic form $\varphi$ in the space of $\pi$ has a nonzero Whittaker-Fourier coefficient, i.e. a Fourier coefficient of
type $\CF^{\psi_{[4]}}(\varphi)$.

First we show that $\CF^{\psi_{[31];v}}(\varphi)$ is nonzero for  isotropic vectors $v$.
Note that the Whittaker character $\psi_{[4]}$ can be given by
$$
\psi_{[4]}(\begin{pmatrix}
1&u&x_{1,1}&x_{1,2}\\
&1&\sqrt{\varsigma}x_{2,1}&x_{2,2}\\
&&1&-\iota(u)\\
&&&1	
\end{pmatrix})=\psi_F(\frac{1}{2}\tr_{E/F}(u)+\alpha x_{2,1}),
$$
for some $\alpha\in F^\times$,
where $x_{2,1}\in \BA_F$ and $u\in \BA_E$.
The Whittaker-Fourier coefficient is given by
$$
\CF^{\psi_{[4]}}(\varphi)(g):=\int_{N_0(F)\bks N_0(\BA)}\varphi(ug)\psi^{-1}_{[4]}(u)\ud u,
$$
where $N_0$ is the unipotent radical of the fixed Borel subgroup $P_0$ of $\RU_{2,2}$.

Denote $N^{(i)}=M_{\hat{i}}\cap N_0$ for $i\in\{1,2\}$, which is the maximal unipotent subgroup of the Levi subgroup $M_{\hat{i}}$.
Since $N_0=N_{\hat{1}}\rtimes N^{(1)}$ and $N^{(1)}$ normalizes the isotropic vector $e_{-2}$,
we may rewrite $\CF^{\psi_{[4]}}(\varphi)$ as the following double integrals
\begin{align*}
\CF^{\psi_{[4]}}(\varphi)(g)=&\int_{N^{(1)}(F)\bks N^{(1)}(\BA)}\psi_{[4]}\vert_{N^{(1)}}(n_2)^{-1}\\
&\ \ \ \ \ \ \ \int_{N_{\hat{1}}(F)\bks N_{\hat{1}}(\BA)}\varphi(n_1 n_2g)\psi_{[31];e_{-2}}(n_1)^{-1}\ud n_1\ud n_2\\
=&\int_{N^{(1)}(F)\bks N^{(1)}(\BA)}\psi_{[4]}\vert_{N^{(1)}}(n_2)^{-1}
\CF^{\psi_{[31];e_{-2}}}(\varphi)(n_2 g)\ud n_2,
\end{align*}
where the character $\psi_{[4]}\vert_{N^{(1)}}$ is the restriction of $\psi_{[4]}$ to the subgroup $N^{(1)}$
and the character $\psi_{[31];e_{-2}}$ is defined in \eqref{eq:psi-3-1}.
It follows that the integrand $\CF^{\psi_{[31];e_{-2}}}(\varphi)(g)$ is not identically zero.

In order to finish the proof of the lemma, we take $v=e_{2}$ in \eqref{eq:psi-3-1}.
Because $q^*(e_{2},e_{2})=0$, $e_{2}$ is isotropic.
Recall that $M_{\hat{1}}(F)$ acts transitively on the set of nonzero isotropic vectors in $\Span_E\{e_2,e_{-2}\}$, which gives the characters of $N_{\hat{1}}(F)\bks N_{\hat{1}}(\BA)$.
Thus
$$
\CF^{\psi_{[31];e_{-2}}}(\pi)\ne 0 \text{ iff }\CF^{\psi_{[31];e_{2}}}(\pi)\ne 0.
$$
More precisely, we have
$$
\psi_{[31];e_{2}}(\begin{pmatrix}
1&u&x_{1,1}&x_{1,2}\\
&1&0&-\iota({x}_{1,1})\\
&&1&-\iota({u})\\
&&&1	
\end{pmatrix} )=\psi_{F}(\frac{1}{2}\tr_{E/F}(x_{1,1})).
$$
Set $N'_2:=N_{\hat{1}}\cap N_{\hat{2}}$, which is the subgroup of $N_{\hat{2}}$ consisting of elements
\begin{equation}\label{eq:N'-2}
\begin{pmatrix}
1&0&x_{1,1}&x_{1,2}\\
&1&0&-\iota({x}_{1,1})\\
&&1&0\\
&&&1	
\end{pmatrix},	
\end{equation}
where $x_{1,2}$ satisfies $\iota({x}_{1,2})+x_{1,2}=0$.
Similarly, as $N_{\hat{1}}=N'_2\rtimes N^{(2)}$ and $N^{(2)}$ normalizes the character $\psi_{[31];e_{2}}\vert_{N'_2}$, we have
\begin{align*}
 \CF^{\psi_{[31];e_{2}}}(\varphi)=& \int_{N^{(2)}(F)\bks N^{(2)}(\BA)} \psi_{[31];e_{2}}\vert_{N^{(2)}}(n_2)^{-1}\\
 &\ \ \ \ \ \ \ \ \int_{N'_2(F)\bks N'_2(\BA)}\varphi(n_1 n_2 g)\psi_{[31];e_{2}}\vert_{N'_2}(n_1)^{-1}
 \ud n_1\ud n_2.
\end{align*}
Remark that the group $N'_2(\BA)$  of $\BA_F$-points (resp. $N'_2(F)$ of $F$-points) consists of matrices of form \eqref{eq:N'-2} with entry $x_{1,1},x_{1,2}\in \BA_E$ (resp. $x_{1,1},x_{1,2}\in E$).
If $\CF^{\psi_{[31];e_{2}}}(\varphi)\ne 0$, then the inner integral
\begin{equation}
\CF_{N'_2}(\varphi)(g):=\int_{N'_2(F)\bks N'_2(\BA)}\varphi(n_1 g)\psi_{[31];e_{2}}\vert_{N'_2}(n_1)^{-1}
 \ud n_1
\end{equation}
is not identically zero.

Since $N^{(1)}$ commutes with $N'_2$,
applying the Fourier expansion along $N^{(1)}$, we have
\begin{equation}\label{eq:F-N'-2}
\CF_{N'_2}(\varphi)(g)=\sum_{\alpha\in F}
\int_{F\bks \BA}\CF_{N'_2}(\varphi)(u_2(z)g)\psi_F(\frac{\alpha}{2} z)^{-1}\ud z,	
\end{equation}
where
\begin{equation}\label{eq:N-(1)}
N^{(1)}(\BA)=(M_{\hat{1}}\cap N_0)(\BA)=\left\{u_2(z):=\begin{pmatrix}
1&0&0&0\\
&1&\sqrt{\varsigma}z&0\\
&&1&0\\
&&&1	
\end{pmatrix}\colon z\in \BA_F\right\}.	
\end{equation}
Combining the integrals over $z\in N^{(1)}$ and  $N'_2$, Equation \eqref{eq:F-N'-2} may be rewritten as
$$
\CF_{N'_2}(\varphi)(g)=\sum_{\alpha\in F}\CF^{\psi_{[2^2];X_\alpha}}(\varphi)(g),
$$
where the character $\psi_{[2^2];X_\alpha}$ is defined in \eqref{eq:psi-2-2} with
$
X_{\alpha}=\begin{pmatrix}
1&\alpha\sqrt{\varsigma}^{-1}\\ 0&-1	
\end{pmatrix}.
$

In fact, it is an easy calculation to have
$$
\begin{pmatrix}
x_{1,1}&\sqrt{\varsigma}y\\ \sqrt{\varsigma}z&-\iota({x}_{1,1})	
\end{pmatrix}
\begin{pmatrix}
1&\alpha\sqrt{\varsigma}^{-1}\\ 0&-1	
\end{pmatrix}=\begin{pmatrix}
x_{1,1}&\alpha x_{1,1}\sqrt{\varsigma}^{-1}-\sqrt{\varsigma}y\\ \sqrt{\varsigma}z &\alpha z+\iota({x}_{1,1})
\end{pmatrix}
$$
where $y,z\in \BA_F$ and $x_{1,1}\in \BA_E$.
Hence
if $\CF_{N'_2}(\varphi)$ is not zero, then there exists $\alpha\in F$ such that
$$
 \CF^{\psi_{[2^2];X_\alpha}}(\varphi)(g)\ne 0.
$$

Note that the adjoint action of the Levi subgroup $M_{\hat{2}}(F)=\GL_2(E)$ on the unipotent radical $N_{\hat{2}}(F)$
is equivalent to its congruent action on the set of Hermitian matrices.
Consider the action of $M_{\hat{2}}(F)$ on the set of characters $\psi_{[2^2];X_{\alpha}}$.
For $g\in \GL_2(E)$, the action of $g$ on $X_{\alpha}$ is given by
$$
g* X_{\alpha}:=g\cdot X_{\alpha}w_2\cdot \iota({g})^t w_2.
$$
Here $X_{\alpha}w_2$ is a Hermitian matrix with discriminant $\det(X_{\alpha}w_2)=1$.
Since the congruent classes of Hermitian matrices are parametrized by the discriminants,
there exists $g\in \GL_2(E)$ such that $g* X_{\alpha}=\vartheta$ (recall $\vartheta=\diag(1,-1)$).
For instance,
$$
\begin{pmatrix}
1&\frac{\alpha}{2\sqrt{\varsigma}}\\ 0&	1
\end{pmatrix}*X_{\alpha}=\begin{pmatrix}
1&0\\ 0&-1	
\end{pmatrix}.
$$
Hence we obtain that
$$
\CF^{\psi_{[2^2];X_\alpha}}(\pi)\ne 0 \text{ if and only if }\CF^{\psi_{[2^2];\vartheta}}(\pi)\ne 0.
$$
Therefore, we conclude that $\CF^{\psi_{[2^2];\vartheta}}(\pi)\ne 0$. This finishes the proof of the lemma.
\end{proof}

By Lemma \ref{lm:U-22}, $\CF^{\psi_{[2^2];\vartheta}}(\pi)$, the space of all functions $\CF^{\psi_{[2^2];\vartheta}}(\varphi)$ with $\varphi$ running through the
space of $\pi$, is nonzero. The functions $\CF_{[2^2];\vartheta}(\varphi)$ are rapidly decreasing automorphic functions on $\RU_{1,1}(\BA)$, the
stabilizer $M_{[2^2];\vartheta}$ of the character $\psi_{[2^2];\vartheta}$ in $M_{\hat{2}}=G_{E/F}(2)$. We are going to consider the Whittaker-Fourier coefficient of
$\CF_{[2^2];\vartheta}(\varphi)$ on $\RU_{1,1}(\BA)$.

Define $N^{(2)}_+$ to be the maximal unipotent subgroup of $M_{[2^2];\vartheta}$, which consists of elements
$
m(\begin{pmatrix}
1&b\\ &	1
\end{pmatrix} )
$
where $b=\iota(b)$.
For $\alpha\in F$, let $\psi_{[2];\alpha}$ be the character of $N^{(2)}_+(\BA)$ defined by
$$
\psi_{[2];\alpha}(m(\begin{pmatrix}
1& b\\ &1	
\end{pmatrix} ))=\psi_F(\alpha b),
$$
where $b\in \BA_F$.
It is trivial on $N^{(2)}_+(F)$.
Remark that the definition of $M_{[2^2];\vartheta}$ implies that $b\in\BA_F$ (see \eqref{eq:U(1,1)}). Hence this character is well-defined.
Recall that for an automorphic form $\varphi'$ on $M_{[2^2];\vartheta}(\BA)$,
its Whittaker-Fourier coefficient  is defined by
\begin{equation}\label{eq:WC-2}
\CF^{\psi_{[2];\alpha}}(\varphi')(h)=\int_{N^{(2)}_+(F)\bks N^{(2)}_+(\BA)}\varphi'(nh)
\psi_{[2];\alpha}^{-1}(n)\ud n.
\end{equation}
When $\alpha=0$, $\CF^{\psi_{[2];0}}(\varphi')$ is the constant term of $\varphi'$.
Since $N^{(2)}_+$ normalizes the character $\psi_{[2^2];\vartheta}$, we can extend $\psi_{[2^2];\vartheta}$ by $\psi_{[2];\alpha}$ to obtain the character $\psi'$ of $N'=N^{(2)}_+N_{\hat{2}}$.
More precisely,
\begin{equation}\label{eq:psi'}
\psi'( \begin{pmatrix}
1&b&x_{1,1}&x_{1,2}\\
&1&x_{2,1}&x_{2,2}\\
&&1&-b\\
&&&1	
\end{pmatrix} ):=\psi_F(\alpha b+\tr_{E/F}\frac{x_{1,1}}{2}),
\end{equation}
where $b$ is in $\BA_F$, $x_{1,1}\in\BA_E$ and $\alpha\in F$. Since $b\in\BA_F$, $N'(\BA)$ is a non-trivial topological subgroup of $N_{0}(\BA)$.
By composing the Whittaker-Fourier coefficient \eqref{eq:WC-2} with
the $[2^2]$-type Fourier coefficient $\CF^{\psi_{[2^2];\vartheta}}$,
we obtain the following  Fourier coefficient of an automorphic form $\varphi$ of $\RU_{2,2}(\BA)$
\begin{align}
\CF_{[2;\alpha]\circ[2^2;\vartheta]}(\varphi):=&\CF^{\psi_{[2];\alpha}}\circ\CF^{\psi_{[2^2];\vartheta}}(\varphi)\nonumber \\
=&\int_{N'(F)\bks N'(\BA)}\varphi(n)\psi'^{-1}(n)\ud n.\label{eq:composition-U}
\end{align}

\begin{lem}\label{lm:2comp22}
If $\pi$ is a generic cuspidal automorphic representation of $\RU_{2,2}(\BA)$, then
there exists $\alpha\in F^\times$ such that
$\CF_{[2;\alpha]\circ[2^2;\vartheta]}(\pi)$ is not zero.
\end{lem}

\begin{proof}
By Lemma \ref{lm:U-22}, for any $\varphi\in \pi$, the function $\CF^{\psi_{[2^2];\vartheta}}(\varphi)$ is nonzero and rapidly decreasing on
$\RU_{1,1}(\BA)$. Hence the space $\CF^{\psi_{[2^2];\vartheta}}(\pi)$ can be realized as a nonzero subspace
in  $L^2(\RU_{1,1}(F)\bks \RU_{1,1}(\BA),\omega_\pi)$,
where $\omega_\pi$ is the central character of $\pi$.
It suffices to show that there exists $\alpha\in F^\times$ such that
the Whittaker-Fourier coefficient of $\CF^{\psi_{[2^2];\vartheta}}(\varphi)$, which is
$\CF^{\psi_{[2];\alpha}}(\CF^{\psi_{[2^2];\vartheta}}(\varphi))$, is not identically zero.

In general, for a nonzero automorphic function $\phi$ of $\RU_{1,1}(\BA)$, belonging to  the space $L^2(\RU_{1,1}(F)\bks \RU_{1,1}(\BA),\omega_\pi)$,
we claim that if $\CF^{\psi_{[2];\alpha}}(\phi)=0$ for all $\alpha\ne 0$, then $\phi$ is equal to $c\omega_\pi$ on $\RU_{1,1}(\BA)$ for some constant $c\in \BC$.

Indeed, suppose that $\CF^{\psi_{[2];\alpha}}(\phi)=0$ for all $\alpha\ne 0$.
Applying the Fourier expansion of $\phi$ along its maximal unipotent subgroup $N^{(2)}_+$, we have
$$
\phi(g)=\sum_{\alpha\in F}\CF^{\psi_{[2];\alpha}}(\phi)(g)=\int_{N^{(2)}_+(F)\bks N^{(2)}_+(\BA)}\phi(ng)\ud n.
$$
It follows that $\phi$ is  invariant under the left multiplication of $N^{(2)}_+(\BA)$.
Of course, $\phi$ is also left invariant under $\RU_{1,1}(F)$. Note that
$N^{(2)}_+(\BA)$ and the Weyl group in $\RU_{1,1}(F)$ generate the subgroup $\SU_{1,1}(\BA)$ of $\RU_{1,1}(\BA)$.
Hence $\phi$ is left $\SU_{1,1}(\BA)$-invariant, which implies $\phi=c\omega_\pi$.

Now, we apply the above claim to the function $\CF^{\psi_{[2];\alpha}}(\CF^{\psi_{[2^2];\vartheta}}(\varphi))$.
Suppose that $\CF^{\psi_{[2];\alpha}}(\CF^{\psi_{[2^2];\vartheta}}(\varphi))$ is zero for all $\alpha\in F^\times$ and $\varphi\in \pi$.
Since $\CF^{\psi_{[2^2];\vartheta}}(\varphi)$ is not zero, we deduce that $\CF^{\psi_{[2^2];\vartheta}}(\varphi)$ must be proportional to $\omega_\pi$ on $\RU_{1,1}(\BA)$.
However, it is known (and is easy to verify) that the Fourier coefficient $\CF^{\psi_{[2^2];\vartheta}}(\varphi)$ is rapidly decreasing
on the Siegel domain of its stabilizer, which is $\RU_{1,1}(\BA)$. Therefore, $\CF^{\psi_{[2^2];\vartheta}}(\varphi)$ can not be proportional to $\omega_\pi$ on $\RU_{1,1}(\BA)$. This contradiction finishes the proof.
\end{proof}

Finally, we complete the proof by using the following lemma.

\begin{lem}\label{lm:comp-31}
Let $\varphi$ be a smooth automorphic function on $\RU_{2,2}(\BA)$ of uniform moderate growth.
Then
$$
\CF_{[2;\alpha]\circ[2^2;\vartheta]}(\varphi)\ne  0
\text{ if and only if }
\CF^{\psi_{[31];v_\alpha}}(\varphi)\ne 0,
$$
where $v_\alpha=e_2+\alpha e_{-2}$ with $\alpha\in F^\times$.
\end{lem}

\begin{proof}
By \eqref{eq:composition-U}, the Fourier coefficient $\CF_{[2;\alpha]\circ[2^2;\vartheta]}(\varphi)$ is an integration over the unipotent subgroup $N'$, 
while by \eqref{eq:psi-3-1}, the Fourier coefficient $\CF^{\psi_{[31];v_\alpha}}(\varphi)$ is an integration over the unipotent subgroup $N_{\hat{1}}$. 
The equivalence of the non-vanishing of the two Fourier coefficients in this lemma is exactly the situation given by \cite[Corollary 7.1]{GRS11} 
(or by \cite[Section 6]{JL-Cogdell} for a more general situation). 
We will use the notation in \cite[Corollary~7.1]{GRS11} in this proof. 

Take $C:=N_{\hat{1}}\cap N'$ to be a topological subgroup of $N_0(\BA)$ consisting of elements of form
$$
\begin{pmatrix}
1&b&x_{1,1}&x_{1,2}\\
&1&0&-\iota(x_{1,1})\\
&&1&-b\\
&&&1	
\end{pmatrix}\qquad (b\in \BA_F,~x_{1,1}\in\BA_E),
$$
where $N'=N^{(2)}_+N_{\hat{2}}$,
and $\psi_C:=\psi'\vert_C$ is the restriction of $\psi'$ (in \eqref{eq:psi'}) to the subgroup  $C(\BA)$.
Write $X=N^{(1)}$ consisting of elements of form $u_2(x)$ (see \eqref{eq:N-(1)}).
Set $Y$ to be the subgroup of $N^{(2)}=N_0\cap M_{\hat{2}}$ consisting of elements
$$
u_{1}(y):=m(\begin{pmatrix}
	1& \sqrt{\varsigma}y\\ &1
\end{pmatrix} ), \qquad (y\in \BA_F).
$$
Note that $N^{(2)}=N^{(2)}_+Y$. It is clear that $X$ and $Y$ are abelian and intersect with $C$ trivially. Meanwhile, we have that
$N_{\hat{1}}=YC$ and $N'=XC$.
Also, it is easy to verify that $X$ and $Y$ normalize $C$. By a simple calculation, we have that the commutator
$$
[u_2(x), u_1(y)]=\begin{pmatrix}
1&0&-\varsigma xy&\sqrt{\varsigma}\varsigma xy^2\\
&1&0&\varsigma xy\\
&&1&0\\
&&&1	
\end{pmatrix}\in C.
$$
Note that the pairing $X(\BA)\times Y(\BA)\to \BC^\times$ given by
$$
(u_2(x), u_1(y))\mapsto \psi_C([u_2(x), u_1(y)])=\psi_F(-\varsigma xy)
$$
is multiplicative in each coordinate and non-degenerate. It follows that
the paring identifies $Y(F)$ with the dual of $X(F)\bks X(\BA)$ and $X(F)$ with the dual of $Y(F)\bks Y(\BA)$.
Thus the data $X$, $Y$, $C$ and $\psi_C$ satisfy all the conditions in \cite[Corollary 7.1]{GRS11}.
Therefore, the proof is completed by directly using \cite[Corollary 7.1]{GRS11}.
\end{proof}

\subsection{On lower rank unitary groups}\label{ssec-lrug}
Take positive integers $a_1$ and $a_2$ such that $a_1+a_2=4$. Consider the endoscopy group $\RU_{E/F}(a_1)\times\RU_{E/F}(a_2)$ of $\RU_{2,2}$.
Let $\phi_4=(\tau_1,1)\boxplus(\tau_2,1)$ be a generic global Arthur parameter of $\RU_{2,2}$.
By Theorem \ref{thm:main-U}, we have the following non-vanishing property.

\begin{cor}\label{cor-nv4}
There exists an automorphic character $\chi$ of $\RU_1$, depending on $\phi_4$ as given above, such that for
any automorphic member $\pi_1$ in the generic global Arthur packet $\wt{\Pi}_{(\tau_1,1)}(\RU_{E/F}(a_1))$ and for any automorphic member $\pi_2$ in the
generic global Arthur packet $\wt{\Pi}_{(\tau_1,1)}(\RU_{E/F}(a_2))$, the product
$$
L(\frac{1}{2},\pi_1\times\chi)\cdot L(\frac{1}{2},\pi_2\times\chi)=
L(\frac{1}{2},\tau_1\times\chi)\cdot L(\frac{1}{2},\tau_2\times\chi)
$$
is nonzero.
\end{cor}

\footnotesize
\noindent\textit{Acknowledgments.}
The research of the first named author is supported in part by the NSF Grant DMS--1600685,
and that of the second named author is supported in part by the National University of Singapore's start-up grant.

\end{document}